\documentclass[12pt, a4paper]{amsart}

\usepackage{amsmath, amsthm, amssymb}
\usepackage{graphicx}
\usepackage[margin = 3cm]{geometry}
\usepackage{multirow, array}
\usepackage[T1]{fontenc}
\usepackage{algo}
\usepackage{float}

\theoremstyle{definition}
\newtheorem{definition}{Definition}
\newtheorem{remark}{Remark}
\theoremstyle{theorem}
\newtheorem{lemma}{\bf Lemma}
\newtheorem{proposition}{\bf Proposition}
\newtheorem{theorem}{\bf Theorem}

\renewcommand\thesection{\Roman{section}}

\makeatletter
	\renewcommand{\@secnumfont}{\bfseries}
    \def\section{\@startsection{section}{1}%
    \z@{.7\linespacing\@plus\linespacing}{.5\linespacing}%
    {\normalfont\bfseries\scshape \centering}}
	\def\subsection{\@startsection{subsection}{2}%
    \z@{.5\linespacing\@plus.7\linespacing}{.3\linespacing}%
    {\normalfont\bfseries}}
	\def\p@subsection{\thesection-}
\makeatother

\begin{document}

\let\MakeUppercase\relax

\title{Parametric robust structured control design}
\author[P. APKARIAN, M. N. DAO, D. NOLL]{P. Apkarian$^*$, M. N. Dao$^{\dag,\ddag}$, D. Noll$^\dag$}

\thanks{$^*$ Control System Department, ONERA, Toulouse, France}
\thanks{$^\dag$ Institut de Math\'ematiques de Toulouse, France}
\thanks{$^\ddag$ Hanoi National University of Education, Vietnam}

\maketitle

\begin{abstract}
We present a new approach to parametric robust controller design, where we
compute controllers of arbitrary order and structure  which 
minimize the worst-case $H_\infty$ norm over a pre-specified set of uncertain parameters. 
At the core of our method is a nonsmooth minimization method tailored to functions which are semi-infinite
minima of smooth functions.  
A rich test bench and a more detailed example illustrate the potential of the technique,
which can deal with complex problems involving multiple possibly repeated uncertain parameters.  

\vspace*{.2cm}
\noindent
{\bf Keywords.}  Real uncertain parameters $\cdot$ structured $H_\infty$-synthesis $\cdot$ 
parametric robust control
$\cdot$ nonsmooth optimization $\cdot$ local optimality $\cdot$ inner approximation 
\end{abstract}

\section{Introduction}
Parametric uncertainty is among the most challenging problems
in control system design due to its NP-hardness.  Albeit, being able to provide solutions to 
this fundamental problem is a must for  any practical design tool worthy of this attribute.
Not surprisingly, therefore, parametric uncertainty has remained high up on the agenda
of unsolved problems in control for the past three decades. 

It is of avail to distinguish between 
analysis and synthesis techniques for parametric robustness.
Analysis refers to assessing  robustness of a closed-loop system when the  controller
is already given.  If the question whether this given controller renders the closed loop
parametrically robustly stable is solved exhaustively, then
it is already an NP-hard problem \cite{TO95}.  Parametric robust synthesis, that is,
computing a controller which is robust against uncertain parameters,  is even harder, because it 
essentially involves an iterative procedure where at every step an analysis problem
is solved. Roughly, we could say that in parametric robust synthesis we have to
optimize a criterion, a single evaluation of which is already NP-hard.

For the analysis of parametric robustness, theoretical and practical tools with only mild conservatism
and acceptable CPUs have been proposed over the years \cite{PDB93}. 
In contrast, no tools with comparable merits in terms of quality and CPU
are currently available for synthesis.  It is fair to say  that the parametric robust synthesis problem has remained open.  
The best currently available techniques for synthesis are the $\mu$ tools going back to
\cite{BDGPS91}, made
available to designers through the MATLAB Robust Control Toolbox. 
These rely on upper bound relaxations of $\mu$ and follow a heuristic 
which alternates
between  analysis and synthesis steps. When it works,  it gives performance and stability certificates,
but the approach may turn out conservative,  
and the computed controllers
are often too complicated for practice. 

The principal obstruction to efficient robust synthesis is the inherent nonconvexity and nonsmoothness of the mathematical 
program
underlying the design. 
These obstacles have to some extent been overcome by the invention of the nonsmooth optimization techniques
 for control \cite{AN06a,AN06b,NPR08}, which we have applied successfully during recent years to multi-model structured
control design
\cite{GA13,Apk13,AN15,AN06a}. These have become available to designers
through synthesis tools like {\tt HINFSTRUCT} or {\tt SYSTUNE}.  
Here we initiate a new line of investigation, which addresses the substantially harder
parametric robust synthesis problem.

In order to understand our approach, it is helpful to
distinguish between inner and outer approximations of the robust control problem
on a set ${\bf \Delta}$ of
uncertain parameters. Outer approximations relax the  problem over 
${\bf \Delta}$ by choosing a larger, but more convenient, set $\widetilde{\bf \Delta}\supset {\bf \Delta}$, the idea being that 
the problem on $\widetilde{\bf \Delta}$  becomes accessible to computations. If solved successfully
on $\widetilde{\bf \Delta}$, this provides performance and robustness certificates for ${\bf\Delta}$.
Typical tools in this class are the upper bound approximation $\overline{\mu}$ of the structured singular value 
$\mu$ developed in \cite{FTD91},  the DK-iteration function {\tt DKSYN} of \cite{RCT2013b}, or LMI-based approaches
like  \cite{SK12,PA06}. 
The principal drawback of outer approximations is the inherent conservatism, which increases significantly with the number of 
uncertainties and their repetitions, and the fact that failures occur more often.
 
Inner approximations are preferred in practice and relax the problem by solving it on a smaller typically finite
subset ${\bf \Delta}_a \subset {\bf \Delta}$. This avoids conservatism and leads to acceptable CPUs,
but has the disadvantage that no immediate stability or performance certificate
for ${\bf \Delta}$ is obtained. Our principal contribution here is to show a way how this
shortcoming  can  be avoided or reduced. We present an efficient technique 
to compute an inner approximation with structured controllers with a local
optimality certificate in such a way  that robust stability and performance are achieved over ${\bf \Delta}$  in the majority of cases.
We then also show how this can be certified a posteriori over ${\bf \Delta}$, when combined with outer
approximation for analysis. The new method we propose is 
termed {\em dynamic inner approximation}, as it generates the inner approximating set ${\bf \Delta}_a$
dynamically. The idea of using inner approximations, and thus multiple models, to solve robust synthesis problems 
is not new and was employed in different contexts,  see e.g. \cite{NSGT99,MGC98,ABKSS93}. 

To address the parametric robust synthesis
problem we use a nonsmooth optimization method
tailored to minimizing a cost function,
which is itself a semi-infinite minimum of smooth functions. This is in contrast with 
previously discussed nonsmooth optimization problems, where a semi-infinite maximum of
smooth functions is minimized, and which have been dealt with successfully in \cite{AN15}. At the core of our new
approach is therefore understanding the 
principled difference between a min-max and a min-min problem, and the 
algorithmic strategies required to solve them successfully. Along with the new synthesis approach, our
key contributions are 
\begin{itemize}
\item an in-depth and rigorous  analysis of worst-case  stability and worst-case  performance problems over a compact parameter range.
\item the description of a new resolution algorithm for worst-case programs along with a proof of convergence in the general nonsmooth case.
\end{itemize}
Note that convergence  to local minima from an arbitrary, even remote, starting point is proved,  as
convergence to global minima is  not algorithmically feasible due to the NP-hardness of the problems.

The paper is organized as follows. Section \ref{sec:robust} states the problem
formally, and
subsection \ref{sec:dynamic} presents our novel dynamic inner approximation technique
and the elements needed to carry it out.
Section \ref{sec:highlight} highlights the principal differences between nonsmooth
min-min and min-max problems. Sections \ref{sec:hinf} and \ref{sec:alpha} 
examine the criteria which arise in the optimization programs, the $H_\infty$-norm,
and the spectral abscissa. Section \ref{sec:algorithm} presents the 
optimization method we designed for min-min problems and the
subsections \ref{conv:hinf}, \ref{conv:alpha} are dedicated to convergence analysis.
Section \ref{sec:testing} presents an assessment and a comparison of our algorithm
on a bench of test examples. Section \ref{sec:missile} gives a more
refined study of a challenging missile control problem.

\section*{Notation}
For complex matrices $X^H$ denotes conjugate transpose. For
Hermitian matrices, $X \succ 0$ means 
positive definite, $X \succeq 0$  positive
semi-definite.   We use concepts from
nonsmooth analysis covered by \cite{Cla83}. For a locally Lipschitz
function $f: \mathbb R^n \rightarrow \mathbb R$, $\partial f (x)$
denotes its (compact and convex) Clarke subdifferential at $x\in \mathbb R^n$. The Clarke directional derivative at $x$ in direction 
$d\in \mathbb R^n$ can be computed as 
\[
f^\circ(x,d)= \max_{g\in \partial f(x)} g^T d\,.
\]

The symbols $\mathcal F_l$, $\mathcal F_u$   denote lower and upper Linear Fractional Transformations (LFT) \cite{ZDG96}.
For partitioned $2\times 2$ block matrices,  $\star$ stands for the Redheffer star product \cite{Red60}.

\section{Parametric robustness}
\label{sec:robust} 
\subsection{Setup}
We consider an  LFT plant in  Fig. \ref{fig-LFT}  with real parametric uncertainties $\mathcal F_u(P, \Delta)$ where 
\begin{equation}\label{plant}
P(s):\left\{
\begin{matrix}
\dot{x}& =& Ax &+& B_pp   &+&B_w w& +& Bu \\
q& =& C_qx&+&D_{qp}p&+&D_{qw}w& +&D_{qu}u\\
z& =& C_zx& +&D_{zp}p & +&D_{zw}w &+ & D_{zu}u \\
y& =& Cx& +& D_{yp}p  &+&D_{yw}w   &+& D u \\
\end{matrix}
\right.
\end{equation}
and $x\in \mathbb R^{n_x}$ is the state, $u \in \mathbb R^{m_2}$ the  control,
$w \in \mathbb R^{m_1}$ the vector of exogenous inputs, 
$y \in \mathbb R^{p_2}$ the output,  and $z\in \mathbb R^{p_1}$ the regulated output.  
The uncertainty channel is defined as $p = \Delta q$ 
where the uncertain matrix
$\Delta$  is without loss assumed to have the block-diagonal  form
\begin{equation}
\label{matrix}
\Delta = {\rm diag}\left[ \delta_1 I_{r_1},\dots, \delta_m I_{r_m}\right]
\end{equation}
with $\delta_1,\dots,\delta_m$ representing real uncertain
parameters, and $r_i$ giving the number of repetitions of $\delta_i$. 
We assume without loss that $\delta = 0$  represents the nominal parameter value.
Moreover, we consider $\delta\in {\bf \Delta}$ in one-to-one correspondence with
the matrix $\Delta$ in (\ref{matrix}).

\begin{figure}[!ht]
\centering
\includegraphics[width = 0.5\columnwidth]{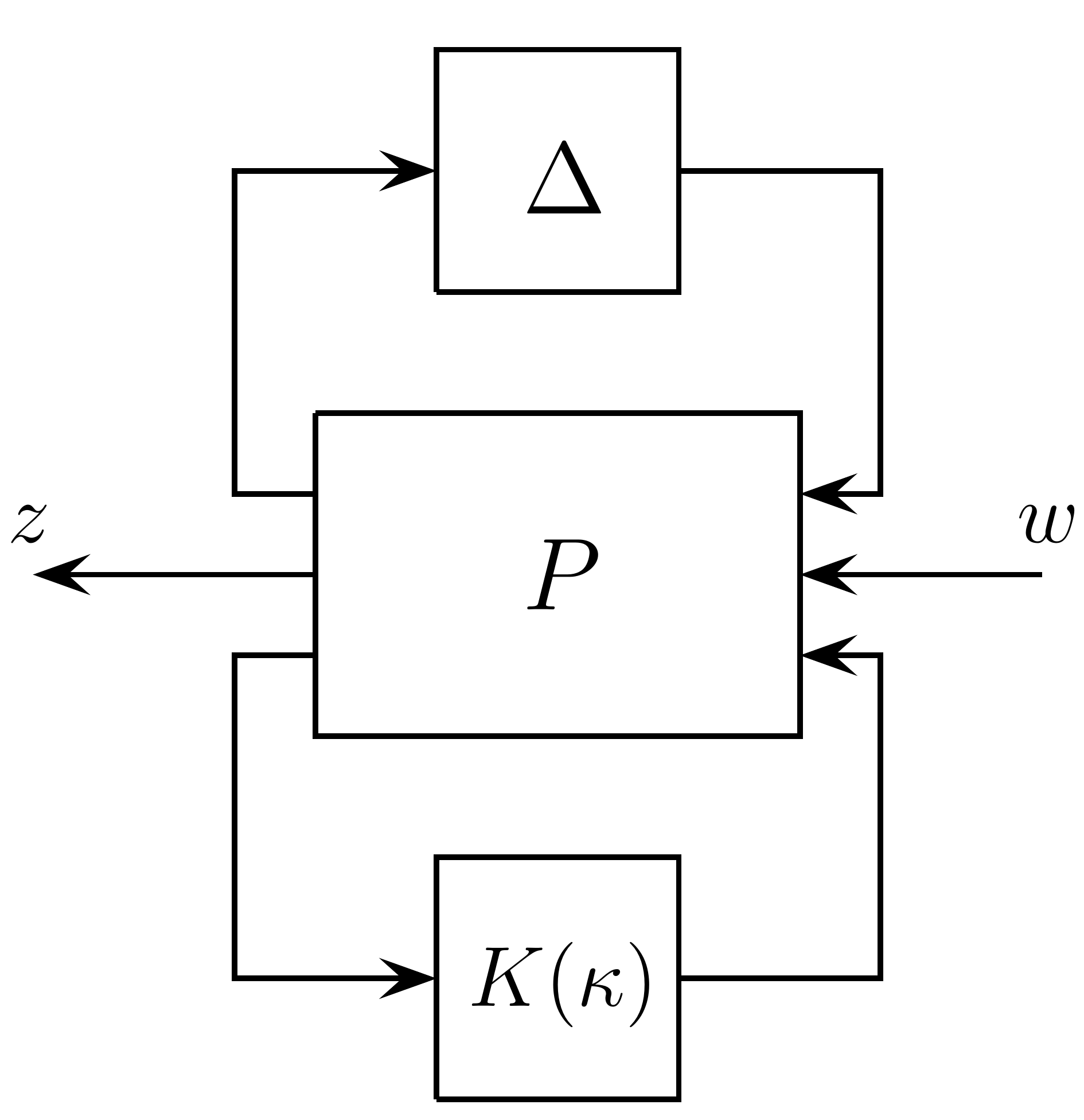}
\caption{Robust synthesis interconnection}
\label{fig-LFT}
\end{figure}

Given a compact convex set ${\bf \Delta}\subset \mathbb R^m$ containing $\delta=0$, the parametric robust
structured $H_\infty$ control problem consists in computing a structured output-feedback controller $u=K(\kappa^*)y$
with the following properties:
\begin{enumerate}
\item[(i)] {\bf Robust stability}. $K(\kappa^*)$ stabilizes $\mathcal F_u(P,\Delta)$ internally for every $\delta \in {\bf \Delta}$.
\item[(ii)] {\bf Robust performance}. Given any other robustly stabilizing controller $K(\kappa)$ with the same structure, 
the optimal controller satisfies 
\[
\displaystyle\max_{\delta\in {\bf \Delta}}\|T_{zw}\left( \delta,\kappa^*\right)\|_\infty \leq \max_{\delta\in {\bf \Delta}}
\|T_{zw}\left(\delta,\kappa\right)\|_\infty.
\]
\end{enumerate}
Here $T_{zw}(\delta,\kappa):=  \mathcal F_l\left( \mathcal F_u (P, \Delta(\delta)),\,K(\kappa) \right)$ denotes the closed-loop transfer function of the performance channel $w\to z$
of (\ref{plant}) when the control loop with controller $K(\kappa)$  and the uncertainty loop with uncertainty $\Delta$ are closed. 

We recall that according to \cite{AN06a} a controller 
\begin{equation}
\label{controller}
K(\kappa): \left\{
\begin{matrix}
\dot{x}_K\!\! &\! \!= \!\!& A_K(\kappa)x_K &\!\!+\!\! & \!\!B_K(\kappa)y \\ 
u\!\! &\!\! = \!\! & C_K(\kappa)x_K &\!\! +\!\! &\!\! D_K(\kappa)y 
\end{matrix}
\right.
\end{equation}
in state-space form
is called {\em structured}  if $A_K(\kappa),B_K(\kappa), \dots$ depend smoothly on a design parameter
$\kappa$ varying in a design space $\mathbb R^n$ or in some constrained subset  
of $\mathbb R^n$. Typical examples of structure include PIDs, reduced-order controllers, observer-based controllers, or 
complex control architectures  combining controller blocks such as set-point filters, feedforward,
washout or notch filters, and much else \cite{AN15}.  In contrast, full-order controllers are state-space 
representations with the same order as $P(s)$ without particular structure, and are sometimes referred to as {\em unstructured},
or as {\em black-box controllers}.

Parametric robust control is among the most challenging problems in linear
feedback control. The structured singular value $\mu$ developed in \cite{ZDG96}
is the principled theoretical tool to describe problem (i), (ii) formally. In the same vein,
based on the spectral abscissa
\[
\alpha(A)=\max\{{\rm Re}(\lambda): \lambda \text{ eigenvalue of } A\}
\]
of a square matrix $A$, criterion (i) may  be written as
\begin{equation}
\label{rob_alpha}
\max_{\delta\in {\bf \Delta}} \alpha \left( A(\delta,\kappa^*) \right) < 0,
\end{equation}
where $A(\delta,\kappa)$ is the  A-matrix of the closed-loop transfer function  $T_{zw}(\delta,\kappa)$. 

If the uncertain parameter set is a cube
${\bf \Delta} = [-1,1]^m$, which is general enough for applications, then the same
information is obtained from the distance to instability in the maximum-norm
\begin{equation}
\label{distance}
d^* = \min \{ \|\delta\|_\infty: \alpha \left( A(\delta,\kappa^*)  \right) \geq 0\},
\end{equation}
because
criterion (i) is now equivalent to $d^* \geq 1$. 
It is known that the computation of any of these elements, $\mu$, (\ref{rob_alpha}),
or (\ref{distance}) is NP-complete, so that their practical use is limited to  analysis of small problems, 
or to the synthesis of tiny ones. 
Practical approaches have to rely on intelligent relaxations, or
{\em heuristics}, which use either inner or outer approximations.

In the next chapters we will develop our
{\em dynamic inner approximation} method to address 
problem (i), (ii).   We solve the  problem on a 
relatively small  set
${\bf \Delta}_a \subset {\bf \Delta}$, which we construct iteratively.

\subsection{Dynamic inner approximation}
\label{sec:dynamic}
The following {\em static} inner approximation to (i), (ii) is
near at hand. After fixing a sufficiently fine approximating static  grid
${\bf \Delta}_s \subset {\bf \Delta}$,  one solves the multi-model $H_\infty$-problem
\begin{equation}
\label{static}
\min_{\kappa \in\mathbb R^n} \max_{\delta \in{\bf \Delta}_s }\|T_{zw}\left( \delta,\kappa\right)\|_\infty .
\end{equation}
This may be addressed with recent 
software tools like {\tt HINFSTRUCT} and {\tt SYSTUNE}, cf. \cite{RCT2013b},  or {\tt HIFOO}  \cite{BHLO06}, but  
has a high computational burden
due to the large number of scenarios in ${\bf \Delta}_s$, which makes it prone to failure. 
Straightforward gridding becomes very quickly intractable for sizable dim$(\delta)$.

Here we advocate a different strategy,
which we call {\em dynamic} inner approximation, because it operates on a substantially smaller 
set ${\bf \Delta}_a\subset {\bf \Delta}$ generated dynamically, whose elements are called the {\em active scenarios}, which
we update a couple of times by applying a search procedure locating problematic parameter scenarios in ${\bf \Delta}$.   
This  leads to a rapidly converging procedure,  much less prone to failure than (\ref{static}).
The method can be summarized as shown in Algorithm \ref{algo1}.

\begin{algorithm}[!ht]
\caption{Dynamic inner approximation for parametric robust synthesis over ~${\bf \Delta}$}\label{algo1}
\begin{algo}
\PARAMETERS $\varepsilon >0$.
\STEP{Nominal synthesis} 
Initialize the set of active scenarios as ${\bf \Delta}_a = ~\{0\}$. 
\STEP{Multi-model synthesis}\label{syn} 
Given the current  finite set ${\bf \Delta}_a \subset {\bf \Delta}$ of active scenarios, compute
a structured multi-model $H_\infty$-controller by solving
\[
v_*=
\min_{\kappa\in \mathbb R^n} \max_{\delta \in {\bf \Delta}_a} \|T_{zw}\left( \delta,\kappa\right)\|_\infty.
\]
The solution is the structured $H_\infty$-controller $K(\kappa^*)$.
\STEP[$\diamond$]{Destabilization} 
Try to destabilize the closed-loop system $T_{zw}\left( \delta,\kappa^*\right)$  by solving the
destabilization problem
\[
\alpha^* = \max_{\delta \in {\bf \Delta} } \alpha \left( A(\delta,\kappa^*) \right).
\]
If {$\alpha^* \geq 0$}, then the solution $\delta^*\in {\bf \Delta}$ destabilizes the loop. 
		Include $\delta^*$ in the active scenarios ${\bf \Delta}_a$ and {go back to} step \ref{syn}.
	If no destabilizing $\delta$ was found
	then {go to} step \ref{perform}.
\STEP{Degrade performance}\label{perform}  
Try to degrade the robust $H_\infty$-performance by solving
\[
v^*=\max_{\delta\in {\bf \Delta}} \|T_{zw} \left(\delta, \kappa^*  \right)\|_\infty.
\]
The solution is $\delta^*$. 
\STEP[$\diamond$]{Stopping test}
If {$v^*< (1+\varepsilon)v_* $}   degradation of performance is only marginal.  Then 
		{exit}, or optionally, {go to} step \ref{post-pro} for post-processing. Otherwise
	 include $\delta^*$ among the active scenarios ${\bf \Delta}_a$ and {go back to} step \ref{syn}.
\STEP[$\diamond$]{Post-processing}\label{post-pro} 
Check robust stability (i) and performance (ii) of $K(\kappa^*)$ over ${\bf \Delta}$ by computing
the distance $d^*$ to instability (\ref{distance}), and its analogue
\(
h^*=\min\{\|\delta\|_\infty: \|T_{zw}(\delta,\kappa^*)\|_\infty \geq v^*\}.
\)
	Possibly use $\mu$-tools from \cite{RCT2013b} to assess $d^*,h^*$
approximately. 
	If {all $\delta^*$ obtained satisfy $\delta^*\not\in  {\bf \Delta}$},
		then {terminate} successfully.
\end{algo}
\end{algorithm}

The principal
elements of Algorithm \ref{algo1} will be analyzed in the following sections. We will focus on the
optimization programs $v^*$ in step 4,  $\alpha^*$ in step 3, and $d^*$, $h^*$ in step 6,
which represent a relatively unexplored  type of nonsmooth programs, with some common features  which we shall put into evidence here. 
In contrast, program
$v_*$ in step 2 is accessible to numerical methods through the work \cite{AN06a} and can be addressed
with tools like {\tt HINFSTRUCT} or {\tt SYSTUNE} available through \cite{RCT2013b}, or {\tt HIFOO} available through \cite{BHLO06}. 
Note that our approach is heuristic in so far as 
we have relaxed (i) and (ii) by computing locally optimal solutions,
so that a global stability/performance  certificate is only provided in the end as a result of  step 6.

\section{Nonsmooth min-max versus min-min programs}
\label{sec:highlight}

\subsection{Classification of the programs in Algorithm \ref{algo1}}
Introducing the functions $a_\pm(\delta) = \pm \alpha\left( A(\delta) \right)$,
the problem of step
3 can be equivalently written in the form
\begin{equation}
\label{minus_alpha}
\begin{array}{ll}
\text{minimize} & a_-(\delta) = - \alpha\left( A(\delta) \right)\\
\text{subject to} & \delta \in {\bf \Delta}
\end{array}
\end{equation}
for a matrix $A(\delta)$ depending smoothly on the parameter $\delta\in \mathbb R^m$.  
Here the dependence of the matrix on controller $K(\kappa^*)$ is omitted for simplicity,
as the latter is fixed in step 3 of the algorithm.
Similarly, if we introduce $h_\pm(\delta)=\pm \| G(\delta)\|_\infty$,  with $G(s,\delta)$ a transfer function depending smoothly
on $\delta\in\mathbb R^m$, then
problem of step 4 has the abstract form
\begin{equation}
\label{minus_h}
\begin{array}{ll}
\text{minimize} & h_-(\delta) = - \| G(\delta)\|_\infty\\
\text{subject to}& \delta \in {\bf \Delta}
\end{array}
\end{equation}
where again controller $K(\kappa^*)$ is fixed in step 4, and therefore
suppressed in the notation.
In contrast,  the $H_\infty$-program $v_*$  in step 2 of Algorithm \ref{algo1}  
has the form
\begin{equation}
\label{h_plus}
\begin{array}{ll}
\text{minimize} & h_+(\kappa) = \|G(\kappa)\|_\infty \\
\text{subject to}& \kappa \in \mathbb R^n
\end{array}
\end{equation}
which is of the more familiar
 min-max type.  Here we use the well-known fact that the $H_\infty$-norm may be written as a 
 semi-infinite maximum function 
 $h_+(\kappa)=\max_{\omega\in [0,\infty]} \overline{\sigma}\left( G(\kappa,j\omega) \right)$.  The maximum
 over the finitely many $\delta \in {\bf \Delta}_a$ in step 2 complies with this structure 
 and may in principle be condensed into
 the form (\ref{h_plus}), featuring only a single transfer $G(s,\kappa)$.  In practice this is  treated as
 in \cite{AN06b}.
  
 Due to the minus sign, 
 programs (\ref{minus_alpha})
and (\ref{minus_h}), written in the minimization form, are now of the novel min-min type, which is given special attention here. This difference is made precise by the following
\begin{definition}
[Spingarn \cite{Spi81}, Rockafellar-Wets \cite{RW98}]
A locally Lipschitz function $f:\mathbb R^n \to \mathbb R$ is lower-$C^1$ at  $x_0\in \mathbb R^n$
if there exist a compact space $\mathbb K$, a neighborhood $U$ of $x_0$, and a mapping $F:\mathbb R^n \times \mathbb K \to \mathbb R$ such that
\begin{equation}
\label{lower}
f(x) = \max_{y\in\mathbb  K} F(x,y)
\end{equation}
for all $x \in U$, and $F$ and $\partial F/\partial x$ are jointly continuous. The function $f$ is said to be upper-$C^1$  if
$-f$ is lower-$C^1$.
\hfill $\square$
\end{definition}

We expect  upper- and lower-$C^1$ functions to behave 
quite differently in descent algorithms. 
Minimization of lower-$C^1$ functions, as required in (\ref{h_plus}),  should lead to a genuinely nonsmooth problem,  
because iterates of a descent method move toward the points of nonsmoothness.
In contrast,
minimization of upper-$C^1$ functions as required in (\ref{minus_alpha}) and (\ref{minus_h}) is expected
to be  better behaved, because iterates move away from
the nonsmoothness. Accordingly, we will want to minimize  upper-$C^1$ functions  in (\ref{minus_alpha}) and (\ref{minus_h})
in much the same way as
we optimize smooth functions in classical nonlinear programming, whereas
the minimization of lower-$C^1$ functions in (\ref{h_plus}) requires specific techniques like 
nonconvex  bundle methods \cite{ANP08,ANP07}. See Fig. \ref{fig-Surface} for an illustration. 

\begin{remark}[Distance to instability]
Note that the computation of the distance to instability $d^*$ defined in  (\ref{distance}) for 
 step $6$ of Algorithm \ref{algo1} has also the features of a min-min optimization program. 
Namely, when written in the form
\begin{equation}
\label{newdist}
\begin{array}{ll}
\text{minimize} & t\\
\text{subject to} & -t \leq \delta_i \leq t,\; i=1,\dots,m\\
&-\alpha\left( A(\delta)\right)\leq 0
\end{array}
\end{equation}
with variable $(\delta,t)\in \mathbb R^{m+1}$,
the Lagrangian of  (\ref{distance}) is
\[
L(\delta,t,\lambda,\mu_\pm)=
t +  \sum_{i=1}^m \mu_{i-} \left( -t-\delta_i \right) + \mu_{i+} (\delta_i-t) -\lambda \alpha\left( A(\delta)  \right)  
\]
for  Lagrange multipliers $\lambda \geq 0$ and $\mu_\pm \geq 0$. In particular, if $(\delta^*,t^*,\lambda^*,\mu_\pm^*)$ is 
a Karush-Kuhn-Tucker (KKT) point \cite{Cla83} of (\ref{newdist}), then the local minimum $(\delta^*,t^*)$ we are looking for is 
also a critical point of the unconstrained program
\[
\min_{\delta\in \mathbb R^m, t\in \mathbb R} L(\delta,t,\lambda^*,\mu_\pm^*),
\]
which features the function $a_-$ and is therefore of min-min type. Therefore, in solving
(\ref{distance}), we expect phenomena of min-min type to surface rather than those of
a min-max program.  A similar comment applies to the computation of $h^*$ in step $6$
of the algorithm.
\end{remark}

\begin{remark}[Well-posedness]
Yet another aspect of Algorithm \ref{algo1}
is that in order to be robustly stable over the parameter set ${\bf \Delta}$, the 
LFTs must be well-posed in the sense that $(I-\Delta\mathcal D)^{-1}$ exists
for every $\delta\in  {\bf \Delta}$, where $\mathcal D$ is the closed-loop D-matrix.
Questioning well-posedness could therefore be included in step 3 of the algorithm,
or added as posterior testing in step 6.
It can be formulated as yet another  min-min program
\begin{equation}\label{eq-wp}
\begin{array}{ll}
\text{minimize} &  -\overline{\sigma}((I- \Delta \mathcal D)^{-1})\\
\text{subject to}& \delta \in {\bf \Delta}
\end{array}
\end{equation}
where one would diagnose the solution $\delta^*$
to represent an ill-posed scenario as soon as  it 
achieves a large  negative value.
Program (\ref{eq-wp}) exhibits the same properties as minimizing $h_-$ in section \ref{sec:hinf} 
and is handled with the same techniques.

For programs $v^*$ in step 4,  $\alpha^*$ in step 3, and $d^*$, $h^*$ in step 6 of Algorithm \ref{algo1},   
well-posedness (\ref{eq-wp}) is a prerequisite.
However, we have observed that it may not be necessary to question well-posedness over ${\bf \Delta}$ at every step,
since questioning stability over ${\bf \Delta}$ has a similar effect. Since the posterior certificate in step 6
of the algorithm covers also well-posedness, this is theoretically justified. 
\end{remark}

\begin{remark}
Our notation makes it easy for the reader to distinguish between
min-min and min-max  programs. Namely, minimizations over the controller variable
$\kappa$ turn out  the min-max ones, while minimizations over the uncertain parameters $\delta$
lead to the min-min type.  
\end{remark}

\subsection{Highlighting the difference between min-max and min-min}
In this section we look at the typical difficulties 
which surface in min-max and min-min programs.  This is crucial for the understanding of our
algorithmic approach. Consider first a min-max program of the form
\begin{equation}\label{eq-minmax}
\min_{\kappa\in\mathbb R^n} \max_{i\in I} f_i(\kappa),
\end{equation}
where the $f_i$ are smooth.
When the set $I$ is finite, we may simply dissolve this into a classical nonlinear programming (NLP)
using one additional dummy variable $t\in\mathbb R$:
\[
\begin{array}{ll}
\text{minimize} & t\\
\text{subject to} &  f_i(\kappa)\leq t,\; i\in I.
\end{array}
\]
The situation becomes more complicated as soon as the set $I$ is infinite, as is
for instance the case in program $v_*$ in step 2 of Algorithm \ref{algo1}. The typical difficulty in min-max programs is to deal
with this semi-infinite character, and one is beholden to use a tailored solution, as for instance 
developed in \cite{AN06a,ANP07,ANP08}. Altogether this type of difficulty is well-known and has been thoroughly studied.

In contrast, a min-min program
\begin{equation}\label{eq-maxmin}
\min_{\delta\in \mathbb R^n} \min_{i\in I} f_i(\delta)
\end{equation}
cannot be converted into an NLP even when $I$ is finite. The problem has disjunctive
character, and if solved to global optimality, min-min programs lead to combinatorial explosion.
On the other hand, a min-min problem has some favorable features when
it comes to solely finding a good local minimum. Namely,
when meeting a nonsmooth iterate  $\delta^j$, where several branches $f_i$ are active,
we can simply pick one of those branches and continue optimization
as if the objective function were smooth. In the subsequent sections we prove that this
intuitive understanding is indeed correct.  Our experimental section will show that good results are obtained if a good heuristic is used. 

The above considerations lead us to introduce the notion of active indices and branches for functions 
$f(\delta)$ defined by the inner $\max$ and $\min$ in (\ref{eq-minmax}) and
(\ref{eq-maxmin}). 

\begin{definition}
{\rm
The set of active indices for $f$  at $\delta$ is defined as
\[
I(\delta):= \left\{ i \in I: f_i(\delta) = f(\delta) \right\}\,.
\]
Active branches of $f$ at $\delta$ are those corresponding to active indices, i.e, $f_i$,  $i \in I(\delta)$.
}
\end{definition}

\begin{figure}[h!]
\centering
\includegraphics[width = 0.5\columnwidth]{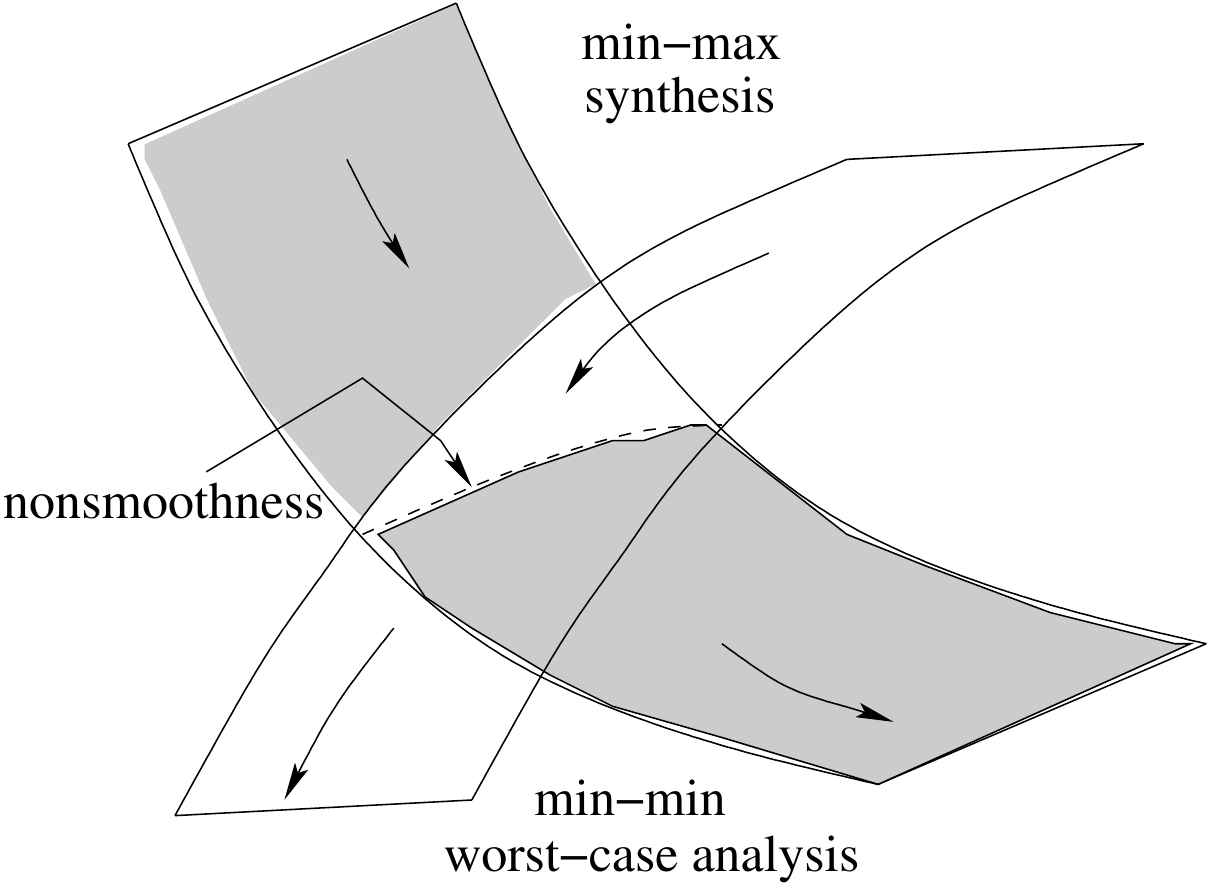}
\caption{Min-max versus min-min programs}
\label{fig-Surface}
\end{figure}

\section{Computing subgradients}
In this section we briefly discuss how the subgradient information needed 
to minimize $h_-$ and $a_-$ is computed.

\subsection{Case of the $H_\infty$-norm}
\label{sec:hinf}
We start by investigating the case of the $H_\infty$-norm $h_\pm$. 
We recall that function evaluation is based on the Hamiltonian algorithm
of \cite{BBK89,BBal90} and its further developments \cite{BSV12}.
Computation of subgradients of $h_-$  in the sense of Clarke can be adapted from
\cite{AN06a}, see also \cite{BB91}. We assume the controller is fixed in this section and
investigate the  properties of $h_-$ as a function of $\delta$.  To this aim, the controller loop is closed 
by substituting the structured controller (\ref{controller}) in (\ref{plant}), and we obtain the transfer function
$M(\kappa):= \mathcal F_l(P,K(\kappa))$. Substantial simplification in Clarke subdifferential computation is 
then obtained by defining the $2\times2$-block transfer function

\begin{equation}\label{eq-starD}
\left[\begin{array}{cc} \ast & T_{qw}(\delta) \\T_{zp}(\delta) & T_{zw}(\delta) \end{array}\right]:= \left[\begin{array}{cc}0 & I \\I & \Delta\end{array}\right] \star M \,,
\end{equation}
where  the dependence on $\kappa$ has now been suppressed, as the controller will
be fixed to $\kappa^*$ after step 2. 
It is readily seen that $T_{zw}$ coincides with the closed-loop transfer function where both controller and uncertainty loops are closed.

Now consider the function $h_-(\delta) := -  \|T_{zw}(\delta)\|_\infty$, 
which is well defined on its domain $\mathbb D := \{\delta \in \mathbb R^m: T_{zw}(\delta) \text{ \rm is internally stable} \}$.
We have the following

\begin{proposition}
\label{prop1}
The function $h_-$  is everywhere Clarke subdifferentiable on $\mathbb D$.
The Clarke subdifferential at $\delta\in \mathbb D$ is the compact and convex set
\begin{eqnarray*}
\partial h_-(\delta) = \bigg\{ \phi_Y: Y =(Y_\omega), \omega\in \Omega(\delta),  \,
Y_\omega \succeq 0, \qquad \\
\textstyle \sum_{\omega\in \Omega(\delta)} {\rm Trace}(Y_\omega)=1 \bigg\}\,,
\end{eqnarray*}
where
the $i$-th entry of $\phi_Y$  is ${\rm Trace}\left( \Delta_i^T \Phi_Y\right)$ with $\Delta_i 
=\partial \Delta/\partial \delta_i$, and
\[
\Phi_Y = - \sum_{\omega\in \Omega(\delta)} {\rm Re}
\left( T_{qw}(\delta,j\omega) P_\omega Y_\omega Q_\omega^H T_{zp}(\delta,j\omega) \right)^T.
\]
Here $\Omega(\delta)$ is  the set of active frequencies at $\delta$, $Q_\omega$ is a matrix whose columns are the left singular vectors associated
with the maximum singular value  of $T_{zw}(\delta,j\omega)$, 
$P_\omega$ is the corresponding matrix of right singular vectors, and $Y_\omega$ is an Hermitian matrix of appropriate size.
\end{proposition}

\begin{proof}
Computation of the Clarke subdifferential of $h_-$ can be obtained from the general
rule $\partial (-h)=-\partial h$, and knowledge of $\partial h_+$, see \cite{AN06a}. 
Note  that in that reference the Clarke subdifferential is with respect to the controller  
and relies therefore  on 
the Redheffer star product 
\[
P \star \left[\begin{array}{cc}K(\kappa) & I \\I & 0\end{array}\right]\,.
\]
Here we apply this in the
upper loop in $\Delta$, so we have to use the analogue expression (\ref{eq-starD}) instead. 
\end{proof}

\begin{remark}
In the case where a single frequency $\omega_0$ is active at $\delta$ and
the maximum singular value  $\overline\sigma$  of $T_{zw}(\delta,j\omega_0)$ has multiplicity 1,
$h_-$ is differentiable at $\delta$ and the gradient is  
\[
\frac{\partial h_-(\delta)}{\partial\delta_i} 
= - {\rm Trace} \, {\rm Re} \, \left( T_{qw}(\delta, j\omega_0) p_{\omega_0 }q_{\omega_0 }^H T_{zp}(\delta,j\omega_0)\right)^T \Delta_i, 
\]
where $p_{\omega_0 }$ and $q_{\omega_0 }$ are the unique right and left singular vectors of 
$T_{zw}(\delta,j\omega_0)$ associated with $\overline{\sigma}(T_{zw}(\delta,j\omega_0))=h_+(\delta)$.
\end{remark}

\begin{proposition}
Let $\mathbb D=\{\delta: T_{zw}(\delta) \text{ \rm  is internally stable}   \}$. Then   $h_+: \delta \mapsto \|T_{zw}(\delta)\|_\infty$
is lower-$C^1$ on   $\mathbb D$, so that $h_-:\delta \mapsto -\|T_{zw}(\delta)\|_\infty$ is upper-$C^1$ there.
\end{proposition}

\begin{proof}
Recall that the maximum singular value has the variational
representation
\[
\overline{\sigma}(G) = \sup_{\|u\|=1} \sup_{ \|v\|=1} \left| u^T G v\right|.
\]
Now observe that $z\mapsto |z|$, being convex,  is lower-$C^1$ as a mapping $\mathbb R^2 \to \mathbb R$, so we may write
it as
\[
|z| = \sup_{l\in \mathbb L} \Psi(z,l)
\]
for $\Psi$ jointly of class $C^1$ and $\mathbb L$ compact. Then
\begin{equation}
\label{4sup}
h_+(\delta) = \sup_{j\omega \in \mathbb S^1} \sup_{\|u\|=1} \sup_{ \|v\|=1} \sup_{l\in\mathbb L}
\Psi\left( u^TT_{zw}(\delta,j\omega)v, l  \right),
\end{equation}
where $\mathbb S^1 = \{j\omega: \omega\in \mathbb R \cup\{\infty\}\}$ is homeomorphic with the $1$-sphere.
This is  the desired representation (\ref{lower}), where the compact space $\mathbb K$
is obtained as $\mathbb K:=\mathbb S^1 \times \{u: \|u\|=1\}\times \{v: \|v\|=1\} \times\mathbb L$, $F$ as 
$F(\delta,j\omega,u,v,l):= \Psi\left( u^TT_{zw}(\delta,j\omega)v, l  \right)$ and $y$ as $y:= (j\omega,u,v,l)$.  
\end{proof}

\subsection{Case of the  spectral abscissa}
\label{sec:alpha}
For the spectral abscissa the situation is more complicated, as $a_\pm$ is 
not locally Lipschitz everywhere. Recall that an eigenvalue $\lambda_i$
of $A(\delta)$ is called active at $\delta$ if ${\rm Re}(\lambda_i) = \alpha\left( A(\delta)\right)$.
We use $I(\delta)$ to denote  the indices of active eigenvalues. Let us write the LFT describing
$A(\delta)$ as $ A(\delta) = \mathcal  A + \mathcal C \Delta (I-\mathcal D \Delta)^{-1} \mathcal B $, 
where dependence on controller parameters
$\kappa$ is again omitted 
and considered absorbed into the state-space data $\mathcal A$, $\mathcal B$, etc.

\begin{proposition}
\label{prop4}
Suppose all active eigenvalues $\lambda_i$, $i\in I(\delta)$ of $A(\delta)$ at $\delta$ are  semi-simple. Then $a_\pm(\delta)= \pm \alpha\left( A(\delta)\right)$
is Clarke subdifferentiable in a neighborhood of $\delta$. The Clarke
subdifferential of $a_-$ at $\delta$ is
$\partial a_-(\delta)= \{ \phi_Y : Y = (Y_i)_{i\in I(\delta)}, Y_i\succeq 0, \sum_{i\in I(\delta)} {\rm Trace}(Y_i)=1 \}$, where the $i$-th entry of $\phi_Y$ is  
$-{\rm Trace}{\,}{\Delta_i}^T \Phi_Y$ with $\Delta_i = \partial \Delta/\partial \delta_i$, and
\[
\Phi_Y = \sum_{i\in I(\delta)} {\rm Re} \left( (I-\mathcal D \Delta)^{-1} 
\mathcal C V_i Y_i U_i^H \mathcal B (I- \Delta \mathcal D)^{-1} \right)^T,
\]
where  $V_i$ is a column matrix of right eigenvectors, $U_i^H$ a row matrix of left eigenvectors
of $A(\delta)$ associated with the eigenvalue  $\lambda_i$,  and  such that $U_i^H V_i = I.$ 
\end{proposition}

\begin{proof}
This  follows from \cite{BO94}. See also \cite{BAN07}. A very concise proof that semi-simple
eigenvalue functions are locally Lipschitz could also be found in \cite{Lui11}. 
\end{proof}

When every active eigenvalue is simple,  $Y_i$ reduces to a scalar $y_i$ and a 
fast implementation is possible. 
We
use the LU-decomposition to solve for $\widetilde u_i$ and  $\widetilde v_i$ in the linear systems
\[
\widetilde u_i^H (I- \Delta \mathcal D): = u_i^H \mathcal B,\;  (I-\mathcal D \Delta) \widetilde v_i := 
\mathcal C v_i\,.
\]
Given the particular structure (\ref{matrix}) of $\Delta$, subgradients with respect to the $k$th entry are readily obtained as a sum 
over $i\in I(\delta)$ of inner products
of the form $ y_i {\rm Re} \,\widetilde{u}_i(J(k))^H  \widetilde{v}_i(J(k))$, 
where  $J(k)$ is a selection of indices associated to the rows/columns of $\delta_k$ in $\Delta(\delta)$. 
Similar inner product forms apply to the computation of $H_\infty$ norm subgradients. 

It was observed in \cite{BO94} that $a_\pm$ may fail to be locally Lipschitz at $\delta$
if $A(\delta)$ has 
a derogatory active eigenvalue. 

\begin{proposition}
\label{prop5}
Suppose every active eigenvalue of $A(\delta)$ is simple. Then $a_-$ is upper-$C^1$
in a neighborhood of $\delta$. 
\end{proposition}
\begin{proof}
If active eigenvalues are simple, then $a_+$ is the maximum of $C^1$ functions in a neighborhood of $\delta$. The result follows from  $a_- = -a_+$. 
\end{proof}

\section{Algorithm for min-min programs}
\label{sec:algorithm}
In this section we present our descent algorithm to solve programs
(\ref{minus_alpha}) and (\ref{minus_h}). We consider an abstract form of the min-min program 
with $f$ a general objective function of this type:
\begin{equation}
\label{min-min}
\begin{array}{ll}
\text{minimize} & f(\delta)\\
\text{subject to}&\delta\in {\bf \Delta}
\end{array}
\end{equation}
where as before ${\bf\Delta}$ is a compact convex set with a convenient structure.
As we already pointed out, the crucial point is that
we want to stay as close as possible to a standard algorithm for smooth optimization, while 
assuring convergence under the specific form of upper nonsmoothness in these programs. 

\begin{algorithm}[!ht]
\caption{Descent method for min-min programs.}\label{algo-outer}
\begin{algo}
\PARAMETERS  $0 < \gamma  < \Gamma < 1$, $0 < \theta < \Theta < 1$.  
\STEP{Initialize} 
Put outer loop counter $j=1$, choose initial guess $\delta^1\in {\bf \Delta}$, 
and fix memory step size $t_1^\sharp > 0$. 
\STEP[$\diamond$]{Stopping}\label{stopping}
If {$\delta^j$ is a KKT point of (\ref{min-min})}
		then {exit}, otherwise
	{go to} inner loop.
\STEP{Inner loop} 
At current iterate $\delta^j$ call the step finding
subroutine (Subroutine \ref{algo-prox}) 
started with last memorized stepsize $t_j^\sharp$
 to find a step $t_k>0$ and a new serious iterate $\delta^{j+1}$  such that 
\[
\rho_k =\frac{f(\delta^j)-f(\delta^{j+1})}{f(\delta^j)-\phi_k^\sharp(\delta^{j+1},\delta^j)} \geq \gamma.
\]
\vspace*{-\baselineskip}
\STEP[$\diamond$]{Stepsize update}
If {$\rho_k\geq \Gamma$} 
		then
		update memory stepsize as $t_{j+1}^\sharp =\theta^{-1} t_k$,
	otherwise
	update memory stepsize as $t_{j+1}^\sharp = t_k$.
	Increase counter $j$ and {go back to} step \ref{stopping}.
\end{algo}
\end{algorithm}

In order to understand
Algorithm \ref{algo-outer} and its step finding subroutine (Subroutine \ref{algo-prox}),  we
recall  from
\cite{Noll12,NPR08} that
\[
\phi^\sharp(\eta,\delta) = f(\delta) + f^\circ(\delta,\eta-\delta)
\]
the standard  model of $f$ at $\delta$, where $f^\circ(\delta,d)$ is the Clarke
directional derivative of $f$ at $\delta$ in direction $d$ \cite{Cla83}.  This model can be thought of as  a substitute for
a first-order Taylor expansion at $\delta$ and can also be represented as
\begin{equation}\label{eq-ClarkeModel}
\phi^\sharp(\eta,\delta) = f(\delta) + \max_{g\in \partial f(\delta)} g^T (\eta-\delta),
\end{equation}
where $\partial f(\delta)$ is the Clarke subdifferential of $f$ at $\delta$.
In the subroutine we
generate lower  approximations $\phi_k^\sharp$ of $\phi^\sharp$
using finite subsets $\mathcal G_k \subset \partial f(\delta)$,
putting
\[
\phi_k^\sharp(\eta,\delta) = f(\delta) + \max_{g\in \mathcal G_k} g^T (\eta-\delta).
\]
We call
$\phi_k^\sharp$ the  working model at inner loop  counter $k$.

\begin{algorithm}[!ht]
\setcounter{algorithm}{0}
\floatname{algorithm}{Subroutine}
\caption{Descent step finding for min-min programs.}\label{algo-prox}
\begin{algo}
\INPUT Current serious iterate $\delta$, last memorized stepsize $t^\sharp>0$. Flag.
\OUTPUT Next serious iterate $\delta^+$.
\STEP{Initialize} 
Put linesearch counter $k=1$, and initialize search at $t_1 =t^\sharp$. 
Choose subgradient $g_0\in \partial f(\delta)$.
Put $\mathcal G_1=\{g_0\}$. 
\STEP{Tangent program}\label{tangent_prog} 
Given $t_k > 0$, a finite set of Clarke subgradients 
$\mathcal G_k\subset \partial f(\delta)$,
and the corresponding  working
model $\phi_k^\sharp(\cdot,\delta) = f(\delta) +\displaystyle \max_{g\in \mathcal G_k} g^T (\cdot-\delta)$,
compute  solution $\eta^k\in {\bf \Delta}$ of the convex quadratic tangent program
\[
{\rm (TP)}\qquad\qquad\qquad\qquad
\min_{\eta\in {\bf \Delta}}
\phi^\sharp_k(\eta,\delta)+\textstyle \frac{1}{2t_k} \|\eta-\delta\|^2. \qquad\qquad\qquad\qquad
\]
\vspace*{-\baselineskip}
\STEP[$\diamond$]{Armijo test} 
Compute
\[
\rho_k = \frac{f(\delta)-f(\eta^k)}{f(\delta)-\phi^\sharp_k(\eta^k,\delta)}
\]
	If {$\rho_k \geq \gamma$} 
		then $\delta^+ = \eta^k$ successfully to Algorithm \ref{algo-outer}. 
	Otherwise
	{go to} step \ref{model_update}
\STEP{If {\tt Flag} $= $ {\tt strict}. Cutting and aggregate plane}\label{model_update} 
Pick a subgradient $g_{k}\in \partial f(\delta)$
such that $f(\delta) + g_k^T (\eta^k-\delta) = \phi^\sharp(\eta^k,\delta)$, or
equivalently, 
$f^\circ(\delta,\eta^k-\delta) = g_{k}^T (\eta^k-\delta)$. Include $g_k$ into
the new set $\mathcal G_{k+1}$ for the next sweep.
Add the aggregate subgradient $g_k^*$ 
into the set $\mathcal G_{k+1}$ to limit its size.
\STEP[$\diamond$]{Step management}\label{manage} 
Compute the test quotient
\[
\widetilde{\rho}_k = \frac{f(\delta) - \phi_{k+1}^\sharp(\eta^k,\delta)}{f(\delta)-\phi_k^\sharp(\eta^k,\delta)}.
\]
	If {$\widetilde{\rho}_k \geq \widetilde{\gamma}$}
		then select $t_{k+1} \in  [\theta t_k,\Theta t_k]$, else
	 keep $t_{k+1} = t_k$. 
	Increase counter $k$ and {go back to} step \ref{tangent_prog}.
\end{algo}
\end{algorithm}

\begin{remark}
Typical values are $\gamma =  0.0001$, $\widetilde{\gamma} = 0.0002$,  and $\Gamma = 0.1$.
For backtracking we use $\theta =\frac{1}{4}$ and $\Theta = \frac{3}{4}$.
\end{remark}

\subsection{Practical aspects of Algorithm \ref{algo-outer}}
The subroutine  
of the descent Algorithm \ref{algo-outer} 
looks complicated, but as we now argue, it reduces to a
standard backtracking linesearch in the majority of cases. To begin with,  if $f$ is certified
upper-$C^1$, then
we completely dispense with step 4 and keep $\mathcal G_k=\{g_0\}$, which by force
reduces the subroutine to a  linesearch
along a projected gradient direction. This is what we indicate by 
{\tt flag} = {\tt upper} in step 4
of the subroutine.

If $f$ is only known to have a strict standard model $\phi^\sharp$ in (\ref{eq-ClarkeModel}), 
without being certified upper-$C^1$, which corresponds to {\tt flag} = {\tt strict},  then
step 4 of the subroutine is needed, as we shall see in section \ref{conv:alpha}. 
However, even then we expect the subroutine to reduce to a standard linesearch.
This is clearly the case when the Clarke subdifferential $\partial  f(\delta)$ at the current iterate $\delta$ is singleton,
because 
$\phi_k^\sharp(\eta,\delta) = f(\delta) + \nabla f(x)^T(\eta-\delta)$ is then again  independent of $k$, 
so $\rho_k \geq \gamma$ reads
\[
f(\eta^k) \leq f(\delta^j) + \gamma \nabla f(\delta^j)^T (\eta^k-\delta^j),
\]
which is the usual Armijo test \cite{Ber82}.  Moreover,
$\eta^k$ is then  a step along the projected gradient $P_{\bf \Delta-\delta}(-\nabla f(\delta))$, which is easy to compute
due to the simple structure of ${\bf \Delta}$. More precisely, for ${\bf \Delta}=[-1,1]^m$
and stepsize $t_k>0$, the solution $\eta$ of tangent program $(TP)$ in step 2 can be computed coordinatewise
as
\[
{\min} \left\{ \gamma_{i} \eta +\textstyle {(2t_k)^{-1}}  \eta^2 + \left(\gamma_{i} - {\delta_i}t_k^{-1}\right)\eta: -1 \leq \eta \leq 1\right\},
\]
where $\gamma_i:=\partial f(\delta)/\partial \delta_i$.
Cutting plane and aggregate plane in step 4 become redundant, and the quotient
$\widetilde{\rho}_k$ in step 5 is also redundant as it is always equal to  $1$. 

\begin{remark}
Step 4 is only fully executed if $f$ is {\em not} certified upper-$C^1$ and 
the subgradient $g_0\in \partial f(\delta)$ in step 1 of Subroutine \ref{algo-prox} does {\em not}
satisfy $f(\delta) + g_0^T (\eta^k-\delta) = \phi^\sharp(\eta^k,\delta)$. In that event step 4 requires computation
of a new subgradient
$g_k\in \partial f(\delta)$ which {\em does} satisfy  $f(\delta) + g_k^T(\eta^k-\delta)=\phi^\sharp(\delta,\eta^k-\delta)$. 
From here on the procedure changes.
The sets $\mathcal G_{k+1}$ may now grow, because we will 
add $g_k$ into $\mathcal G_{k+1}$. This corresponds to what happens in a bundle method.
The tangent program $(TP)$ has now to be solved numerically using a QP-solver, but
since we may limit
the number of elements of $\mathcal G_{k+1}$ using the idea of the aggregate subgradient 
of Kiwiel \cite{Kiw83}, see also \cite{Dao}, this is still very fast.   
\end{remark}

\begin{remark}
For the spectral abscissa $f(\delta) = a_-(\delta)$, which is not certified upper-$C^1$,
we use this cautious variant, where the computation of  $g_k$ in step 4 may be required. 
For $f=a_-$ this leads to a low-dimensional semidefinite program.
\end{remark}

\begin{remark}
The stopping test in step 2 of Algorithm \ref{algo-outer} can be delegated to
Subroutine  \ref{algo-prox}.
Namely, if $\delta^j$ is a Karush-Kuhn-Tucker point
of (\ref{min-min}), then $\eta^k = \delta^j$ is solution of the tangent program $(TP)$. This means we can
use the following practical stopping tests:
If the inner loop at iterate $\delta^j$ finds $\delta^{j+1}\in {\bf \Delta}$
such that
\[
\frac{\|\delta^{j+1}-\delta^j\|}{1 + \|\delta^j\|} < {\rm tol}_1,
\quad
\frac{|f(\delta^{j+1})-f(\delta^j)|}{1+|f(\delta^j)|} < {\rm tol}_2,
\]
then we decide that $\delta^{j+1}$ is optimal and stop. That is, the $(j+1)$st inner loop is not started.
On the other hand, if the inner loop at $\delta^j$ has difficulties finding a new iterate 
and provides five consecutive unsuccessful backtracks $\eta^k$ such that
\[
\frac{\|\eta^{k}-\delta^j\|}{1 + \|\delta^j\|} < {\rm tol}_1,
\quad
\frac{|f(\eta^{k})-f(\delta^j)|}{1+|f(\delta^j)|} < {\rm tol}_2,
\]
or if a maximum $k_{\rm max}$ of linesearch steps $k$ is exceeded,
then we decide that 
$\delta^j$ was already optimal and stop. In our experiments we use
${\rm tol}_1 = 10^{-4}$, ${\rm tol}_2 = 10^{-4}$, $k_{\rm max}=50$.
\end{remark}

\begin{remark}
The term {\em stepsize} used for the parameter $t$ in the tangent program $(TP)$ in step 2
of Algorithm \ref{algo-prox} is understood verbatim when $\mathcal G_k$ consists of a single element $g_0$
and the minimum in $(TP)$ is unconstrained, because then $\|\eta^k-\delta\| = t_k \|g_0\|$. However,
even in those cases where step 4 of the subroutine is carried out in its full version, 
$t_k$ still acts like a stepsize in the sense that decreasing $t_k$ gives smaller steps (in the inner loop), while  increasing  $t^\sharp$
allows larger steps (in the next inner loop).
\end{remark}

\subsection{Convergence analysis for the negative $H_\infty$-norm}
\label{conv:hinf}
Algorithm \ref{algo-outer} was studied in much detail in \cite{Noll12}, and we review the convergence result here,
applying them directly to the functions $a_-$ and $h_-$.
The significance of the class of upper-$C^1$ functions for convergence lies in the following

\begin{proposition}
Suppose $f$ is upper-$C^1$ at $\bar{\delta}$. Then its standard model $\phi^\sharp$ is strict at $\bar{\delta}$
in the following sense:  For every $\varepsilon > 0$ there exists $r > 0$ such that
\begin{equation}
\label{strict}
f(\eta) \leq \phi^\sharp(\eta,\delta) + \varepsilon \|\eta-\delta\|
\end{equation}
is satisfied for all $\delta,\eta\in B(\bar{\delta},r)$.
\end{proposition}
\begin{proof}
The following, even stronger property of upper-$C^1$ functions was proved in \cite{Dao}, see
also \cite{Noll10,Noll12}.
Suppose $\delta^k\to \bar{\delta}$ and $\eta^k\to \bar{\delta}$, and let $g_k\in \partial f(\delta_k)$ arbitrary.  Then there exist
$\varepsilon_k\to 0$ such that 
\begin{equation}
\label{stronger}
f(\eta^k) \leq f(\delta^k) + g_k^T(\eta^k-\delta^k) + \varepsilon_k \|\eta^k-\delta^k\|
\end{equation}
is satisfied.
\end{proof}
Remark \ref{note} below
shows that upper-$C^1$, and thus (\ref{stronger}), are stronger than strictness (\ref{strict}) of the standard model. 

\begin{theorem}[Worst-case $H_\infty$ norm on $\bf \Delta$]
\label{theorem1}
Let $\delta^j \in {\bf \Delta}$ be the sequence generated by Algorithm
{\rm \ref{algo-outer}} with standard linesearch for minimizing  program {\rm (\ref{minus_h})}. 
Then the sequence $\delta^j$ converges to a 
Karush-Kuhn-Tucker point $\delta^*$ of {\rm (\ref{minus_h})}.
\end{theorem}

\begin{proof}
The proof of \cite[Theorem 2]{NPR08}  shows that every accumulation point of the sequence
$\delta^j$ is a critical point of (\ref{minus_h}), provided $\phi^\sharp$ is strict. Moreover, since the 
iterates are feasible, we obtain a KKT point. See Clarke \cite[p. 52]{Cla83} for a definition. 
However, it was observed in \cite{Dao} that 
estimate (26) in that proof can be replaced
by  (\ref{stronger})  when the objective is upper-$C^1$. Since this is the case
for $h_-$ on its domain $\mathbb D$, 
the step finding Subroutine \ref{algo-prox}  can be
reduced to a linesearch.
Reference \cite{Dao} gives also details on how to deal with the 
constraint set ${\bf \Delta}$. Note that hypotheses assuring boundedness of the sequence $\delta^j$ in 
\cite{NPR08,Noll12,Dao} are not needed, since ${\bf \Delta}$ is bounded.

Convergence to a single KKT point is now assured through  \cite[Cor. 1]{Noll12}, 
because  $G$ depends analytically on $\delta$, 
so that $h_-$ is  a subanalytic function, and 
satisfies therefore the \L ojasiewicz inequality \cite{BDL06}. 
Subanalyticity of $h_-$ can be derived from the following fact \cite{BM88}.
If $F:\mathbb R^n \times \mathbb K \to \mathbb R$ is subanalytic, and $\mathbb K$ is subanalytic
and compact, then $f(\delta)= \min_{y\in \mathbb K} F(\delta,y)$ is subanalytic. We apply this to the negative of  
(\ref{4sup}).
\end{proof}

\begin{remark}
\label{note}
The lightning function $f:\mathbb R \to \mathbb R$ 
in \cite{KK02} is an example which has a strict standard model but is not upper $C^1$. 
It is Lipschitz 
with constant $1$ and has $\partial f(x) = [-1,1]$ for every $x$. The standard model of
$f$ is strict, because for all $x,y$ there exists $\rho = \rho(x,y)\in [-1,1]$ such that
\begin{multline*}
f(y) = f(x) + \rho |y-x|  \leq f(x) + {\rm sign}(y-x)(y-x) \\
\leq f(x) + f^\circ(x,y-x)= \phi^\sharp(x,y-x),
\end{multline*}
using the fact that sign$(y-x)\in \partial f(x)$.  At the same time $f$ is certainly not upper-$C^1$, 
because it is not semi-smooth in the sense of \cite{Mif77}.
This shows that the class of functions $f$ with a strict standard model offers
a scope of its own, justifying the effort made in the step finding subroutine.
\end{remark}

\subsection{Convergence analysis for the negative spectral abscissa}
\label{conv:alpha}
While we obtained an ironclad convergence certificate for the
$H_\infty$-programs (\ref{minus_h}), and similarly, for (\ref{eq-wp}), theory is more complicated with
program (\ref{minus_alpha}). 
In our numerical testing $a_-(\delta)=-\alpha\left( A(\delta)\right)$ 
behaves consistently like an upper-$C^1$ function, and 
we expect this to be true at least  if all active eigenvalues
of $A(\delta^*)$ are semi-simple.  
We now argue that we expect $a_-$ to have a strict standard model  as a rule.

Since $A(\delta)$ depends analytically on $\delta$, the eigenvalues
are roots of a characteristic polynomial $p_\delta(\lambda) = \lambda^m + a_1(\delta) \lambda^{m-1} + \dots + a_m(\delta)$
with coefficients $a_i(\delta)$ depending  analytically on  $\delta$. For fixed $d\in \mathbb R^m$,
every eigenvalue $\lambda_\nu(t)$ of $A(\delta^*+td)$ has therefore
a Newton-Puiseux expansion of the form
\begin{equation}
\label{puis}
\lambda_\nu(t) = \lambda_\nu(0) + \sum_{i=k}^\infty \lambda_{\nu,i-k+1}t^{i/p}
\end{equation}
for certain $k,p\in\mathbb N$, where the coefficients $\lambda_{\nu,i}=\lambda_{\nu,i}(d)$ and leading exponent $k/p$
can be determined by the Newton polygon \cite{MBO97}. If all active eigenvalues
of $a_-(\delta)=-\alpha(A(\delta))$ are semi-simple, then $a_-$ is Lipschitz around $\delta^*$ 
by Proposition \ref{prop4}, 
so that necessarily $k/p \geq 1$ in (\ref{puis}). It then follows
that either $a_-'(\delta^*,d)=0$ when $k/p > 1$ for all active $\nu$, or $a_-'(\delta^*,d) = -{\rm Re}\, \lambda_{\nu,1}\leq a_-^\circ(\delta^*,d)$ for  the
active $\nu \in I(\delta^*)$ if $k/p=1$. In both cases $a_-$ satisfies the strictness estimate (\ref{strict}) {\em directionally}, and we expect
$a_-$ to have a strict standard model. Indeed, for $k/p=1$ we have
$a_-(\delta^*+td) \leq a_-(\delta^*) + a_-^{\circ}(\delta^*,d) t  - {\rm Re}\,\lambda_{\nu,2}t^{(p+1)/p} + {\rm o}(t^{(p+1)/p})$, while the case $k/p>1$
gives $a_-'(\delta^*,d)=0$, hence
$a_-^\circ(\delta^*,d) \geq 0$, and so
$a_-(\delta^*+td) \leq a_-(\delta^*)- {\rm Re}\, \lambda_{\nu,1}  t^{k/p} + {\rm o}(t^{k/p})
\leq a_-(\delta^*) + a_-^\circ(\delta^*,d)t - {\rm Re}\, \lambda_{\nu,1}t^{k/p} + {\rm o}(t^{k/p})$. As soon as these estimates hold uniformly over
$\|d\|\leq 1$, $a_-$ has indeed a strict standard model, i.e., 
we have the following

\begin{lemma}
\label{lem}
Suppose every active eigenvalue of $A(\delta^*)$ is semi-simple, and suppose the following two conditions are satisfied:
\begin{equation}
\label{sups}
\begin{array}{l}
\displaystyle
\lim_{t\to 0}
\sup_{\|d\|\leq 1} \sup_{\nu\in I(\delta^*), k/p=1} \sum_{i=k+1}^\infty {\rm Re}\, \lambda_{\nu,i-k+1}(d)t^{i/p-1} \geq 0 \\
\displaystyle\lim_{t\to 0}
\sup_{\|d\|\leq 1} \sup_{\nu\in I(\delta^*), k/p>1} \sum_{i=k}^\infty {\rm Re}\, \lambda_{\nu,i-k+1}(d)t^{i/p-1} \geq 0.
\end{array}
\end{equation}
Then the standard model of  $a_-$ is strict at $\delta^*$. 
\hfill $\square$
\end{lemma}

Even though these conditions are not easy to check, they seem to be verified most of the time, so that
the following result reflects what we observe in practice
for the min-min program of the negative spectral abscissa $a_-$. 

\begin{theorem}[Worst-case spectral abscissa on $\bf \Delta$]
\label{theorem2}
Let $\delta^j\in {\bf \Delta}$ be the sequence generated by Algorithm
{\rm \ref{algo-outer}} for program {\rm (\ref{minus_alpha})}, where the step finding subroutine is carried out
with step {\rm 4} activated. Suppose every accumulation point $\delta^*$ of the sequence $\delta^j$ is simple or semi-simple and satisfies condition
{\rm (\ref{sups})}. Then  the sequence converges  
to a unique KKT point of program {\rm (\ref{minus_alpha})}.
\end{theorem}

\begin{proof}
We apply once again \cite[Corollary 1]{Noll12}, using the fact that
$a_-$ satisfies the \L ojasiewicz inequality at all accumulation points.
\end{proof}

\begin{remark} 
Convergence certificates for minimizing $a_-$ or $a_+$ 
seem to hinge on additional hypotheses which are hard to
verify in practice. 
In  \cite{BLO02} the authors propose  the gradient sampling algorithm
to minimize $a_+$, and their subsequent convergence analysis in \cite{BLO05} needs at least
local Lipschitzness of $a_+$, which is observed in practice but difficult to verify
algorithmically. A similar comment applies to the hypotheses of Theorem \ref{theorem2}, which appear to be 
satisfied in practice, but remain difficult to check directly.\end{remark}

\subsection{Multiple performance measures}
Practical applications often feature several design requirements  combining $H_\infty$ and
$H_2$ performances with spectral constraints related to pole locations. The results in section \ref{conv:hinf}
easily extend to this case upon defining 
$H(\kappa,\delta) := \max_{i\in I} h_i\left(T_{z_i,w_i}(\kappa,\delta)\right)$, where
several performance channels $w_i \to z_i$ are assessed
against various requirements  $h_i$,  as  in \cite{AGB14,Apk13}. 
All results developed so far carry over to multiple requirements, because 
the worst-case  multi-objective performance in step 4 of Algorithm \ref{algo1} 
involves $H_-=-H$ which has the same min-min structure as before. 

\newcommand{\vvline}{\vrule width 1pt}
\newcommand{\hhline}{\noalign{\hrule height 1pt}}

\section{Experiments}
\subsection{Algorithm testing}
\label{sec:testing}
In this section our dynamic inner  relaxation technique (Algorithm \ref{algo1}) is tested on a bench of  $14$ 
examples of various sizes and structures.
All test cases have been taken and adapted from the literature and are described in Table \ref{tab-list}. 
Some tests have been made more challenging by adding uncertain parameters in order to 
illustrate the potential of the technique for higher-dimensional parametric domains ${\bf \Delta}$. 
The notation  $[r_1\; r_2\; \dots \; r_m]$ in the rightmost column of the table stands for 
the block sizes in $\Delta = {\rm diag}\left[ \delta_1 I_{r_1},\dots, \delta_m I_{r_m}\right]$. 
Uncertain parameters have been normalized so that ${\bf \Delta} = [-1,1]^m$, and the nominal value
is $\delta = 0$. 

The dynamic relaxation technique of {Algorithm \ref{algo1}} is first compared to static relaxation (\ref{static}). 
That technique consists  in choosing a dense enough static grid ${\bf \Delta}_s$
of the uncertainty box ${\bf \Delta}$ and to 
perform a multi-model synthesis for a large number card$({\bf\Delta_s})$
of models. In consequence, static relaxation cannot be considered a practical
approach. Namely,

\begin{itemize}
\item Dense grids become quickly intractable  for high-dimensional ${\bf \Delta}$.
\item Static relaxation may lead to overly optimistic answers in terms of worst-case performance  
if critical parametric configurations are missed by gridding.
\end{itemize}
This is what is observed in Table \ref{tab-comp},  where we have used a $5^m$-point grid with 
$m= {\rm dim}(\delta)$ the number of uncertain parameters. 
Worst-case performance is missed in tests $6$, $9$, $12$ and $14$, as we
verified by {Algorithms $2$}.
Running times may rise to hours or even days for cases $1$, $2$, $5$ and $10$. 
On the other hand, when gridding turns out right, then {Algorithm \ref{algo1}} and static relaxation are equivalent. 
In this respect, the dynamic relaxation of {Algorithm \ref{algo1}} can be regarded as a cheap, and therefore very successful,  
way to cover the uncertainty box. The number of scenarios in ${\bf \Delta}_a$ rarely exceeds  $10$  in our testing.
Computations were performed using Matlab R2013b on OS Windows 7 Home Premium with CPU Intel Core i5-2410M
running at 2.30 Ghz and 4 GB of RAM. 

The results achieved by {Algorithm \ref{algo1}} can be certified
{\em a posteriori} 
 through the mixed $\mu$ upper bound \cite{FTD91}. 
This technique computes an overestimate $\overline\mu_p$ of the worst-case performance on the unit cube
${\bf \Delta}$. We introduce
the ratio $\rho := \overline \mu_p/ h_\infty$, where $h_\infty$ is the underestimate of $\mu$
predicted by our  Algorithm \ref{algo1}, given in column 4 of Table
\ref{tab-comp}. Clearly $h_\infty \leq \bar{\mu}_p$, or what is the same, $\rho \geq 1$,
so that values $\rho\approx 1$  
certify the values  predicted by Algorithm \ref{algo1}. 
Note that a value $\rho \gg 1$ indicates failure
to certify the value $h_\infty$ a posteriori, but such a failure could be due {\em either} to
a sub-optimal result of Algorithm \ref{algo1}, {\em or} to conservatism of
the upper bound $\bar{\mu}_p$. 
This was {\em not} observed  in our present testing,
so that  Algorithm \ref{algo1} was certified in all cases. 
For instance, in row 2 of Table \ref{tab-dksyn} we have a guaranteed
performance for parameters in ${\bf \Delta}/1.01$. 

Our last comparison is between Algorithm \ref{algo1}
and {\tt DKSYN} for complex and real $\mu$ synthesis, 
and the results are shown in Table \ref{tab-dksyn}. A value $\overline{\mu}_{\mathbb R}=b$ in column 6
of that table
means worst-case performance of $b$ is guaranteed
on the cube $(1/b)\,[-1,1]^m$. 
It turned out that
no reasonable certificates were to be obtained with $\overline{\mu}_\mathbb R$ synthesis, 
since $b=\overline{\mu}_\mathbb R \gg 1$ as a rule,
 so that $(1/b)[-1,1]^m$ became too small to be of use,  except for test cases $8$, $9$ and $13$.  
 In this test bench, Algorithm \ref{algo1} achieved better worst-case performance 
on a larger uncertainty box with simpler controllers. It also proves competitive
in terms of execution times. 

\begin{table}[h]
\caption{Test cases \label{tab-list} }
\centering
\setlength{\tabcolsep}{2pt}
\begin{tabular}{!\vvline c !\vvline c | c | c | c !\vvline}

\hhline
N\textsuperscript{o} & Benchmark name & Ref. & States & Uncertainty block structure \\ \hhline  
1 & Flexible Beam & \cite{DFT92} & 8 & [1\; 1\; 1\; 3\; 1] \\ \hline
2 & Mass-Spring-Dashpot & \cite{Res91} & 12 & [1\; 1\; 1\; 1\; 1\; 1] \\ \hline
3 & DC Motor & \cite{CK09} & 5 & [1\; 2\; 2] \\ \hline
4 & DVD Drive & \cite{FSBS03} & 5 & [1\; 3\; 3\; 3\; 1\; 3] \\ \hline
5 & Four Disk & \cite{Enn84} & 10 & [1\; 3\; 3\; 3\; 3\; 3\; 1\; 1\; 1\; 1] \\ \hline  
6 & Four Tank & \cite{VGD01} & 6 & [1\; 1\; 1\; 1] \\ \hline
7 & Hard Disk Drive & \cite{GPK05} & 18 & [1\, 1\, 1\, 2\, 2\, 2\, 2\, 1\, 1\, 1\, 1] \\ \hline
8 & Hydraulic Servo & \cite{CD94} & 7 & [1\; 1\; 1\; 1\; 1\; 1\; 1\; 1] \\ \hline
9 & Mass-Spring System & \cite{ACAGF99} & 4 & [1\; 1] \\ \hline
10 & Tail Fin Controlled Missile & \cite{Kru93} & 23 & [1\; 1\; 1\; 6\; 6\; 6] \\ \hline
11 & Robust Filter Design 1 & \cite{SK08} & 4 & [1] \\ \hline
12 & Robust Filter Design 2 & \cite{Kim06} & 2 & [1\; 1] \\ \hline
13 & Satellite & \cite{NTA04} & 5 & [1\; 6\; 1] \\ \hline
14 & Mass-Spring-Damper & \cite{RCT2013b}  & 8 & [1]  \\ \hhline
\end{tabular}
\end{table}%

\begin{table}[h]
\centering
\caption{Comparisons of  Algorithm \ref{algo1} with static relaxation  on unit box  running times in sec.,  I: intractable\label{tab-comp} }
\setlength{\tabcolsep}{2pt}
\begin{tabular}{!\vvline c !\vvline r !\vvline r | r | r !\vvline r | r | r !\vvline}
\hhline
\multirow{2}{*}{N\textsuperscript{o}} & \multicolumn{1}{c!\vvline}{\multirow{2}{*}{order}} & \multicolumn{3}{c!\vvline}{Algorithm \ref{algo1}} 
& \multicolumn{3}{c!\vvline}{Static relaxation} \\ \cline{3-8}
& & \multicolumn{1}{c|}{\scriptsize\# scenarios} & \multicolumn{1}{c|}{\scriptsize $H_\infty$ norm} & \multicolumn{1}{c!\vvline}{\scriptsize time} 
& \multicolumn{1}{c|}{\scriptsize\# scenarios} & \multicolumn{1}{c|}{\scriptsize $H_\infty$ norm} & \multicolumn{1}{c!\vvline}{\scriptsize time} \\ \hhline 
1 & 3 & 4 &      1.290 &     25.093 & 3125 & I & $\infty$ \\ \hline
2 & 5 & 16 &      2.929 &    261.754 & 15625 & I & $\infty$ \\ \hline
3 & PID & 2 &      0.500 &      6.256 & 125 &      0.500 &    127.952 \\ \hline
4 & 5 & 1 &     45.455 &      2.012 & 15625 &     45.454 &   4908.805 \\ \hline
5 & 6 & 6 &      0.672 &     68.768 & 9765625 & I & $\infty$ \\ \hline  
6 & 6 & 4 &      5.571 &     41.701 & 625 &      5.564 &   3871.898 \\ \hline
7 & 4 & 4 &      0.026 &     34.647 & 48828125 & I & $\infty$ \\ \hline
8 & PID & 3 &      0.701 &     10.140 & 390625 & I & $\infty$ \\ \hline
9 & 4 & 4 &      0.814 &     22.917 & 25 &      0.759 &     67.268 \\ \hline
10 & 12 & 6 &      1.810 &    159.299 & 15625 & I & $\infty$ \\ \hline
11 & 4 & 4 &      2.636 &     16.723 & 5 &      2.636 &      6.958 \\ \hline
12 & 1 & 3 &      2.793 &      8.221 & 25 &      2.660 &     23.400 \\ \hline
13 & 6 & 5 &      0.156 &     48.445 & 125 &      0.156 &    876.039 \\ \hline
14 & 5 & 3 &      1.651 &     39.250 & 5 &      1.644 &     27.456 \\ \hhline
\end{tabular}
\end{table}%

\begin{table}[h]
\centering 
\caption{Comparisons between {\tt DKSYN} (complex and real) $\mu$ synthesis and 
dynamic relaxation  on the same uncertainty box  \label{tab-dksyn}} 
\setlength{\tabcolsep}{2pt}
\begin{tabular}{!\vvline c !\vvline r | r | r !\vvline r | r | r !\vvline  c !\vvline}
\hhline
\multirow{2}{*}{N\textsuperscript{o}} & \multicolumn{3}{c!\vvline}{complex $\mu$ syn.} 
& \multicolumn{3}{c!\vvline}{real $\mu$ syn.} & \multicolumn{1}{c!\vvline}{Algorithm \ref{algo1}} \\ \cline{2-8}
& \multicolumn{1}{c|}{ord.} & \multicolumn{1}{c|}{$\overline{\mu}_{\mathbb C}$} & \multicolumn{1}{c!\vvline}{time} 
& \multicolumn{1}{c|}{ord.} & \multicolumn{1}{c|}{$\overline{\mu}_{\mathbb R}$} & \multicolumn{1}{c!\vvline}{time}
& \multicolumn{1}{c!\vvline}{$\rho=\bar{\mu}_p/h_\infty$}\\ \hhline 
1 & 38 &      2.072 &     80.231 & 88 &      1.835 &     86.144   & 1.00 \\ \hline
2 & 54 &      2.594 &     123.288 & 66 &      2.586 &    141.181 &  1.01 \\ \hline
3 & 51 &     17.093 &     76.799 & 65 &     16.854 &     27.269 &  1.01\\ \hline
4 & 5 &     72.464 &     27.113 & 5 &     45.455 &     53.898 & 1.00\\ \hline
5 & 10 &      5.151 &    131.259 & 10 &      1.894 &    315.949 &  1.01\\ \hline  
6 & 6 &      4.558 &     17.519 & 12 &      4.555 &     29.469 &   1.01\\ \hline
7 & 18 &     50.451 &    159.152 & F & F & F  & 1.01 \\ \hline
8 & 61 &      0.963 &    100.636 & 61 &      0.878 &    133.740  & 1.05 \\ \hline
9 & 24 &      0.921 &     47.565 & 28 &      0.989 &    112.820 &  1.04\\ \hline
10 & 147 &      5.639 &   1412.402 & 337 &      2.834 &   7611.679   & 1.04\\ \hline
11 & 14 &      1.804 &     13.759 & 14 &      1.782 &     22.293   & 1.02\\ \hline
12 & 10 &      2.268 &     16.021 & 16 &      2.323 &     21.310 &  1.01 \\ \hline
13 & 133 &      0.821 &    183.052 & 255 &      0.509 &    257.589 &  1.04 \\ \hline
14 & 14 &      1.523 &     16.583 & 16 &      1.562 &     46.722 & 1.00 \\ \hhline
\end{tabular}
\end{table}%

\subsection{Tail fin controlled missile}
\label{sec:missile}
We now illustrate
our robust synthesis technique in more depth for a tail fin controlled missile. 
This problem is adapted from \cite[Chapter IV]{Kru93} and has been made more challenging by adding parametric uncertainties
in the most critical parameters. 
The linearized rigid body dynamics of the missile are 
\[
\begin{array}{l l l}
\begin{bmatrix} \dot\alpha \\ \dot q \end{bmatrix}
& =\begin{bmatrix}
  Z_\alpha   &  1  \\
  M_\alpha   & M_q   
\end{bmatrix}
\begin{bmatrix} \alpha \\ q \end{bmatrix}
& +\begin{bmatrix} Z_d \\ M_d \end{bmatrix} u \\ 
\begin{bmatrix} \eta \\ q \end{bmatrix}
& =\begin{bmatrix}
  V/kG \, Z_\alpha   & 0 \\
  0 &   1   
\end{bmatrix}
\begin{bmatrix} \alpha \\ q \end{bmatrix}
& +\begin{bmatrix} V/kG \, Z_d \\ 0 \end{bmatrix} u 
\end{array}
\]
where $\alpha$ is the angle of attack, $q$ the pitch rate, $\eta$ the vertical acceleration and $u$ the fin deflection.
Both $\eta$ and $q$ are measured through appropriate devices as described below. A more realistic model also includes bending  modes of the missile structure. In this application, we have $3$ bending modes whose contribution to 
$\eta$ and $q$ is additive and described as follows:
\[
\begin{bmatrix} \eta_i(s) \\q_i(s) \end{bmatrix} = \frac{1}{s^2+2\zeta \omega_i s +  \omega_i^2}
\begin{bmatrix} s^2 \Xi_{\eta_i} \\  s \Xi_{q_i} \end{bmatrix}, \; i = 1,2,3 \,.
\]

It is also important to account  for actuator and detector dynamics. The actuator is modeled as a 2nd-order transfer function with
damping $0.7$ and natural frequency $188.5$ rad./sec. Similarly, the accelerometer and pitch rate gyrometer are 2nd-order transfer functions with damping $0.7$ and natural frequencies $377$ rad./sec. and $500$ rad./sec., respectively. 

Uncertainties affect both rigid and flexible dynamics and the deviations from nominal are 
30\% for $Z_\alpha$, 15\% for $M_\alpha$, 30\% for $M_q$, and 10\% for each $\omega_i$.
This leads to an uncertain model with uncertainty structure given as
\[
\Delta = {\rm diag}\left[ \delta_{ Z_\alpha},\delta_{ M_\alpha},\delta_{ M_q},\delta_{ \omega_1}I_6,
\delta_{ \omega_2}I_6,\delta_{ \omega_3}I_6\right] \,,
\]
which corresponds to $\delta\in \mathbb R^6$
and repetitions $[1\;1\;1\;6\;6\;6]$ in the terminology of Table \ref{tab-list}.
The controller structure includes both feed-forward $K_{\rm ff}(s)$ and feedback $K_{\rm fb}(s)$ actions 
\[
u_c =  K_{\rm ff}(s) \eta_r + K_{\rm fb}(s) \begin{bmatrix} \eta_r-\eta_m \\ -q_m \end{bmatrix} =
 K(s)  \begin{bmatrix} \eta_r-\eta_m \\q_m \\ \eta_r \end{bmatrix}\,,
\]
where $\eta_r$ is the acceleration set-point and $\eta_m$, $q_m$ are the detectors outputs. 
The total number of design parameters $\kappa$ in $K(\kappa,s)$
is ${85}$, as a tridiagonal state space representation of a $12$-th order controller was used.

The missile autopilot is optimized over $\kappa\in \mathbb R^{85}$ to meet the following requirements:
\begin{itemize}
\item The acceleration $\eta_m$ should track the reference input $\eta_r$ 
with a rise time of about $0.5$ seconds. In terms of
the transfer function  from $\eta_r$ to the tracking error  $e:=\eta_r-\eta_m$ 
this is expressed as $||W_e(s)T_{e\eta_r}||_\infty \leq 1$,  where the weighting
function $W_e(s)$ is 
\[
W_e(s):= 1/M\frac{s/\omega_B + M}{s/\omega_B + A},\; A = 0.05,\, M = 1.5,\, \omega_B = 10 \,.
\]
\item Penalization of the high-frequency rate of variation of the control signal and roll-off are captured through the constraint
$||W_u(s)T_{u\eta_r}||_\infty \leq 1$, where $W_u(s)$ is a high-pass weighting
$W_u(s):=\left(s/100(0.001 s + 1)\right)^2$.

\item Stability margins at the plant input are specified through the  $H_\infty$ constraint $\|W_o(s) S(s) W_i(s)\|_\infty \leq 1$, 
where $S$ is the input sensitivity function
$S:= (I+K_{\rm fb}G)^{-1}$ and with static weights $W_o = W_i = 0.4$. 
\end{itemize}

Finally, stability and performance requirements must hold for the entire range of parametric uncertainties, 
where $\bf \Delta$ is 
   the $\mathbb R^{6}$-hyperbox with limits in percentage given above. The resulting nonsmooth program 
$v^*$ to be solved in step 2 of Algorithm \ref{algo1} takes the form 
\[
\min_{\kappa\in \mathbb R^{85}} \max_{\delta \in {\bf \Delta}_a\subset \mathbb R^6} \|T_{zw}\left( \delta,\kappa\right)\|_\infty.
\]
We have observed experimentally
that controllers $K(s)$ of order  greater than $12$ do not improve much. The order of the augmented plant 
including flexible modes, detector and actuator dynamics, and
weighting filters is $n_x=23$. 

The evolution of the worst-case $H_\infty$ performance vs. iterations in {Algorithm \ref{algo-outer}} 
(and its {Subroutine \ref{algo-prox}}) 
is problem-dependent. 
For the missile example, a destabilizing  uncertainty is found at the $1$st iteration. 
 The algorithm then settles very quickly in $5$
 iterations on a 
 final set ${\bf \Delta}_a$ consisting of $6$ scenarios. The number of scenarios in the final ${\bf \Delta}_a$ coincides with the number of iterations in Algorithm \ref{algo1} plus the nominal scenario, and can be seen in column 3 of Table \ref{tab-comp}.
Note that the evolution  of the worst-case $H_\infty$ performance is not always monotonic. Typically 
the curve may bounce back when a bad parametric configuration $\delta$ is discovered by the algorithm. 
This is the case e.g.  for the mass-spring example. 

The achieved values of the $H_\infty$ norm and corresponding running times are given in Table \ref{tab-comp}. 
Responses to a step reference input for $100$ models from the uncertainty set  ${\bf \Delta}$ 
are shown in Fig. \ref{fig-TimeSimu} to validate  the robust design. 
Good tracking is obtained over the entire parameter range. 
The magnitude of the  $3$ controller gains  of  $K(s)$ are plotted in Fig. \ref{fig-SigmaK}.   
Robust roll-off and notching of flexible modes are clearly achieved. Potential issues due to pole-zero cancellations are avoided as a consequence
of allowing parameter variations in the model. Finally, Fig. \ref{fig-Nichols} displays the Nichols plots  for $100$ models sampled in the uncertainty  set. We observe that good  "rigid" margins as well as  attenuation 
of the flexible modes over $\bf \Delta$ has been achieved.  

\begin{remark}
Real $\mu$ synthesis turned out  time-consuming, exceeding two hours in the missile example. 
The controller order inflates to $337$ and conservatism is still present as compared to dynamic relaxation
via Algorithm \ref{algo1}. 
A value $\overline{\mu}_{\mathbb R} = 2.834$ reads as a worst-case $H_\infty$ performance of $2.834$ 
over the box ${\bf \Delta} = 1/2.834\, [-1,1]^m$.   To resort to interpreting
uncertain parameters as complex cannot  be considered an acceptable 
workaround  either. 
Even when it delivers a result,  this approach 
as a rule leads to  high-order controllers ($147$ states in the missile example).
Complex $\mu$ synthesis is also fairly conservative, as we expected. 
It appears that scaling- or multiplier-based approaches
using outer relaxations \cite{Young96,PDB93}  encounter two
typical difficulties:
\begin{itemize}
\item The number and repetitions of parametric uncertainties  lead
to conservatism. 
\item Repetitions of the parameters lead to high-order multipliers, which in turn produce high-order controllers. 
\end{itemize}
Our approach is not affected by these issues.
\end{remark}

\begin{remark}
Static relaxation remains intractable even for a coarse grid of $5$ points in each dimension. See Table \ref{tab-comp}. 
\end{remark}

\begin{figure}[htp]
\centering
\includegraphics[width = 0.8\columnwidth]{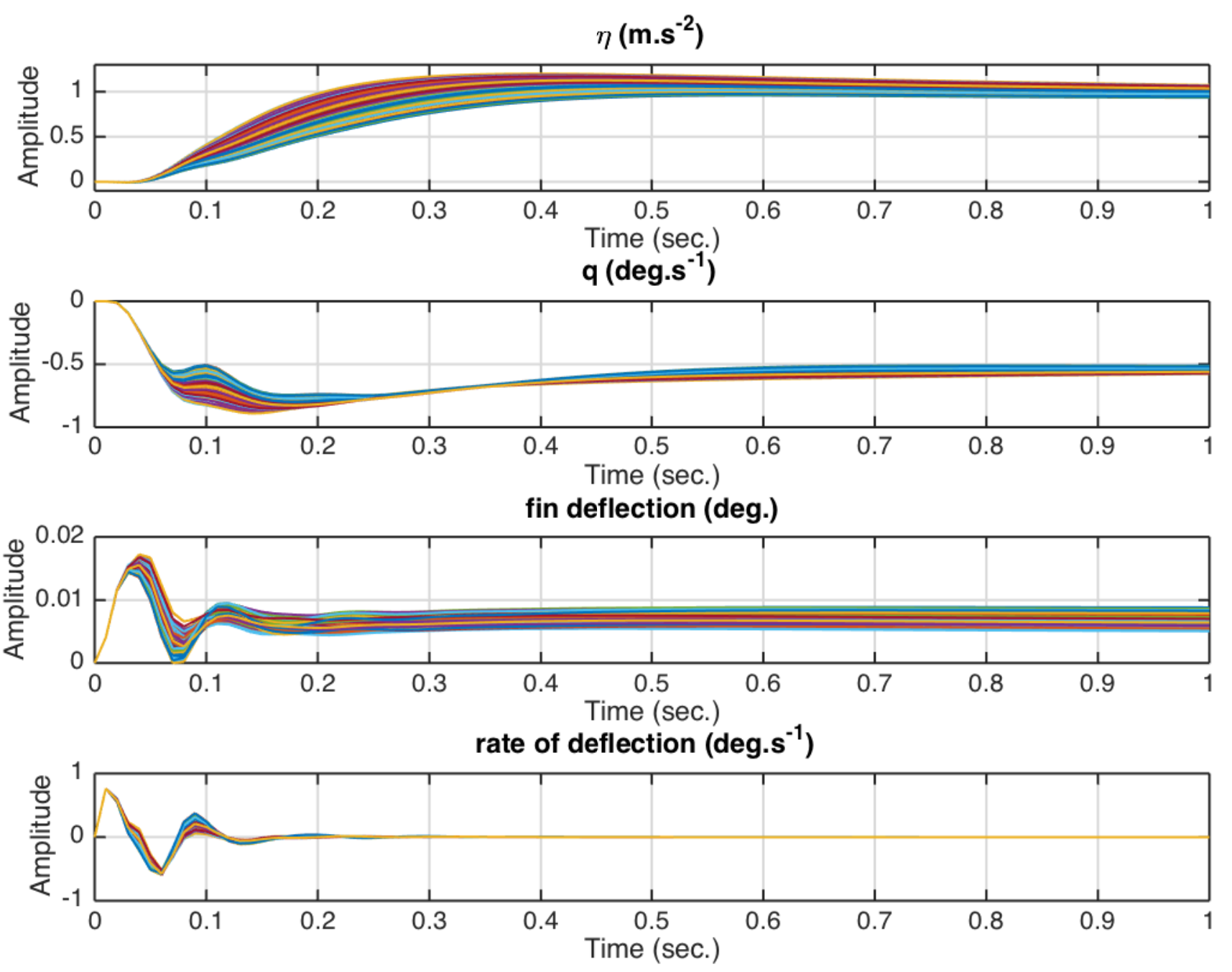}
\caption{Step responses of controlled missile for $100$ sampled models in uncertainty range} 
\label{fig-TimeSimu}
\end{figure}

\begin{figure}[htp]
\centering
\includegraphics[width = 0.8\columnwidth]{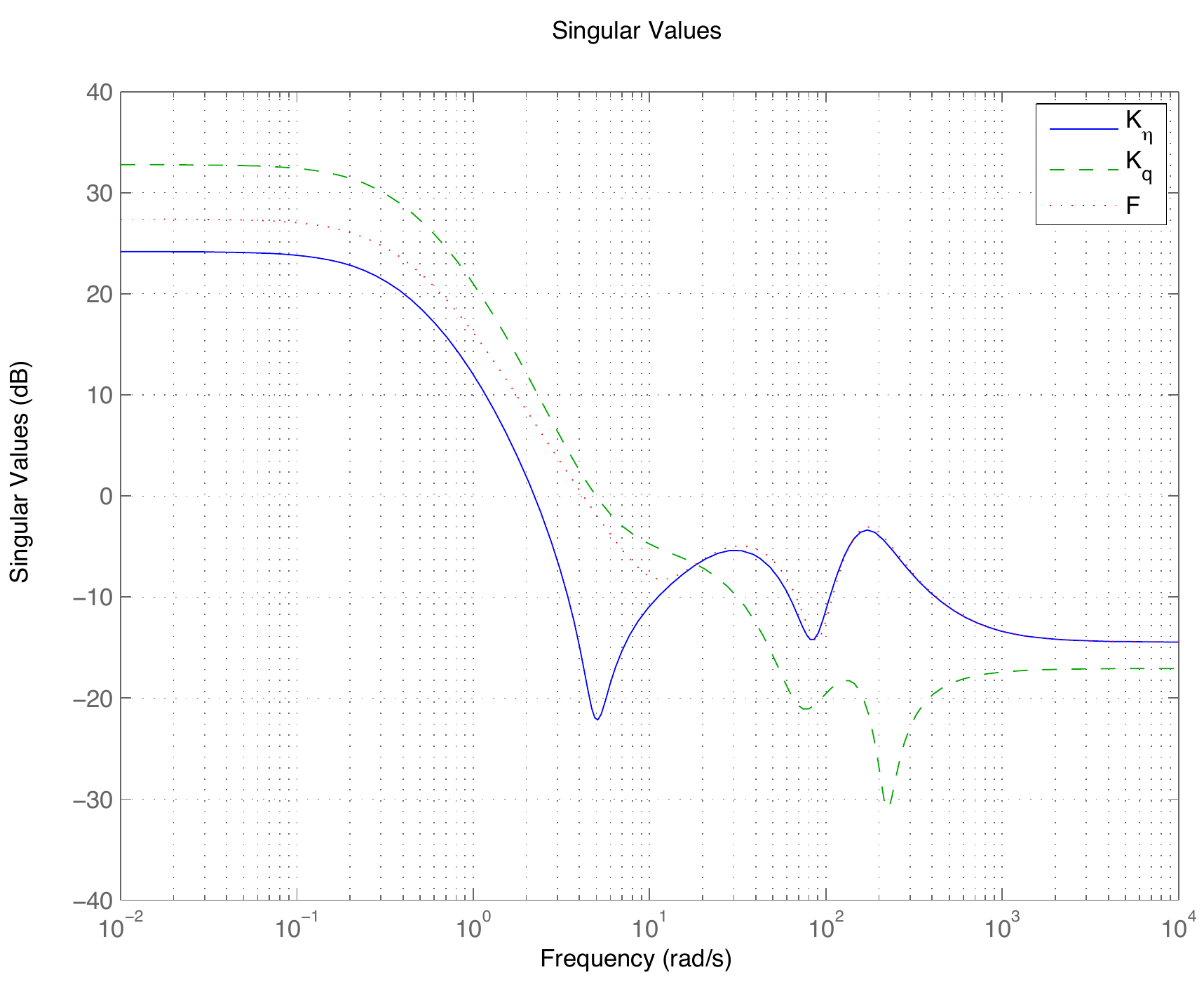}
\caption{Feedback and feed-forward gains} 
\label{fig-SigmaK}
\end{figure}

\begin{figure}[htp]
\centering
\includegraphics[width = 0.8\columnwidth]{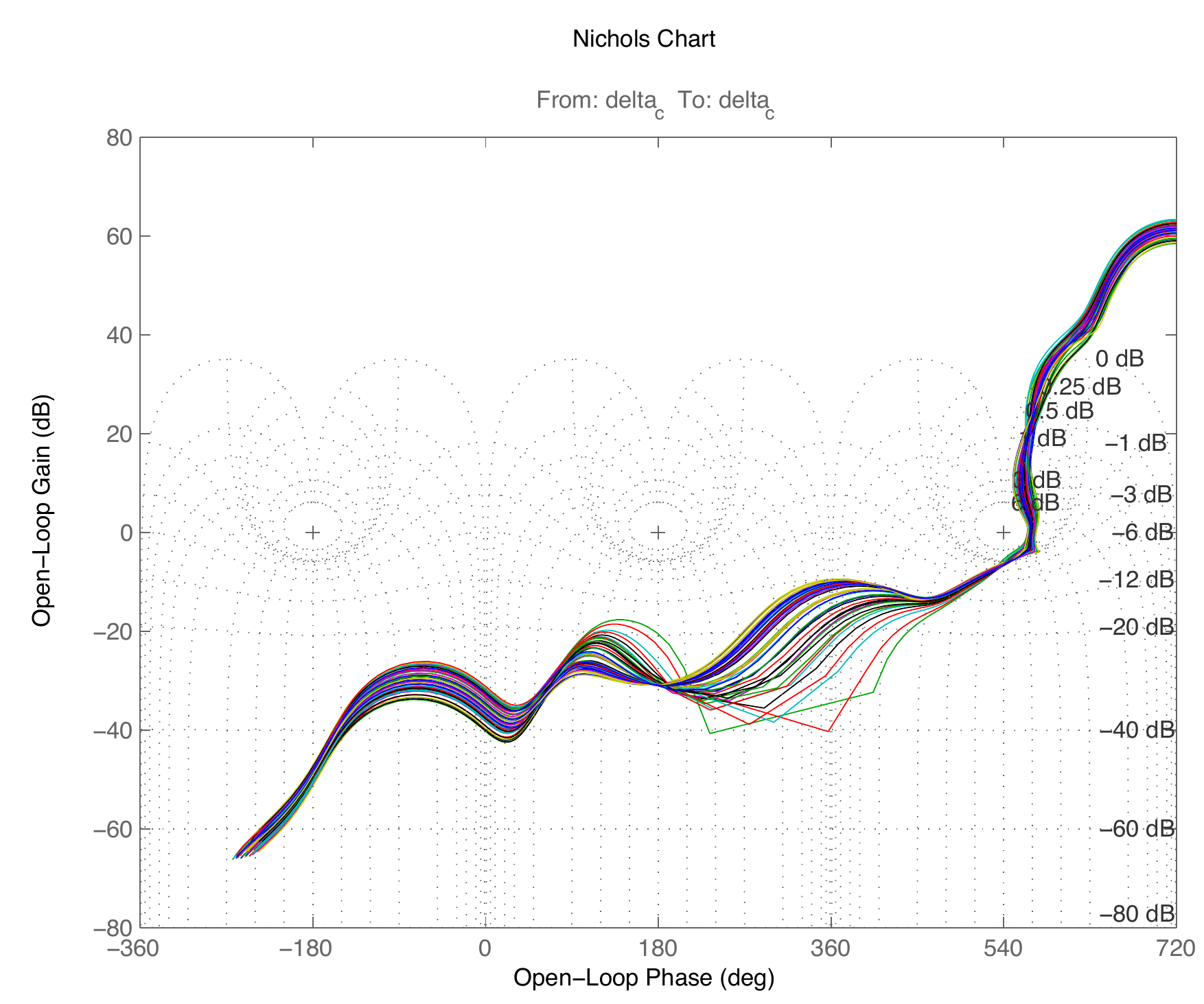}
\caption{Nichols plots for $100$ sampled models in uncertainty range} 
\label{fig-Nichols}
\end{figure}

\section{Conclusion}
We have presented a novel algorithmic approach to
parametric robust $H_\infty$ control  with structured controllers.
A new inner relaxation technique termed {\em dynamic inner approximation},
adapting a set of parameter scenarios $\Delta_a$ iteratively,
was developed and shown to work rapidly without introducing
conservatism. Global robustness and performance certificates are then best
obtained {\em a posteriori} by applying analysis tools based on outer approximations.
At the core our new method is leveraged by sophisticated nonsmooth optimization techniques tailored to
the class of upper-$C^1$ stability and performance functions.  
The  approach was tested on a bench of
challenging examples, and within a case study. 
The results indicate that the proposed technique is a valid practical tool,
capable of solving challenging design 
problems with parametric uncertainty.

\bibliographystyle{IEEEtran}
\bibliography{}

\begin{thebibliography}{10}
\providecommand{\url}[1]{#1}
\csname url@samestyle\endcsname
\providecommand{\newblock}{\relax}
\providecommand{\bibinfo}[2]{#2}
\providecommand{\BIBentrySTDinterwordspacing}{\spaceskip=0pt\relax}
\providecommand{\BIBentryALTinterwordstretchfactor}{4}
\providecommand{\BIBentryALTinterwordspacing}{\spaceskip=\fontdimen2\font plus
\BIBentryALTinterwordstretchfactor\fontdimen3\font minus
  \fontdimen4\font\relax}
\providecommand{\BIBforeignlanguage}[2]{{%
\expandafter\ifx\csname l@#1\endcsname\relax
\typeout{** WARNING: IEEEtran.bst: No hyphenation pattern has been}%
\typeout{** loaded for the language `#1'. Using the pattern for}%
\typeout{** the default language instead.}%
\else
\language=\csname l@#1\endcsname
\fi
#2}}
\providecommand{\BIBdecl}{\relax}
\BIBdecl

\bibitem{TO95}
H.~{\"{O}zbay}. O.~Toker, ``On the {$\mathcal{NP}$}-hardness of the purely
  complex {$\mu$} computation, analysis/synthesis, and some related problems in
  multidimensional systems,'' in \emph{Proc.\ American Control Conf.}, Seattle,
  June 1995, pp. 447--451.

\bibitem{PDB93}
A.~Packard, J.~C. Doyle, and G.~J. Balas, ``Linear, multivariable robust
  control with a {$\mu$} perspective,'' \emph{J. Dyn. Sys., Meas., Control,
  Special Edition on Control}, vol. 115, no.~2b, pp. 426--438, June 1993.

\bibitem{BDGPS91}
G.~J. Balas, J.~C. Doyle, K.~Glover, A.~Packard, and R.~Smith,
  \emph{{$\mu$}-Analysis and synthesis toolbox: {U}ser's {G}uide}.\hskip 1em
  plus 0.5em minus 0.4em\relax {The MathWorks, Inc.}, 1991.

\bibitem{AN06a}
P.~Apkarian and D.~Noll, ``Nonsmooth {$H_\infty$} synthesis,'' \emph{IEEE
  Trans. Automat. Control}, vol.~51, no.~1, pp. 71--86, 2006.

\bibitem{AN06b}
------, ``Nonsmooth optimization for multidisk {$H_\infty$} synthesis,''
  \emph{Eur. J. Control}, vol.~12, no.~3, pp. 229--244, 2006.

\bibitem{NPR08}
D.~Noll, O.~Prot, and A.~Rondepierre, ``A proximity control algorithm to
  minimize nonsmooth and nonconvex functions,'' \emph{Pac. J. Optim.}, vol.~4,
  no.~3, pp. 571--604, 2008.

\bibitem{GA13}
P.~Gahinet and P.~Apkarian, ``Automated tuning of gain-scheduled control
  systems,'' in \emph{Proc.\ IEEE Conf.\ on Decision and Control}, Florence,
  December 2013, pp. 2740 -- 2745.

\bibitem{Apk13}
P.~Apkarian, ``Tuning controllers against multiple design requirements,'' in
  \emph{Proc.\ American Control Conf.}, Washington, June 2013, pp. 3888 --
  3893.

\bibitem{AN15}
P.~Apkarian and D.~Noll, ``Optimization-based control design techniques and
  tools,'' in \emph{Encyclopedia of Systems and Control}, J.~Baillieul and
  T.~Samad, Eds.\hskip 1em plus 0.5em minus 0.4em\relax Springer-Verlag, 2015.

\bibitem{FTD91}
M.~K.~H. Fan, A.~L. Tits, and J.~C. Doyle, ``Robustness in the presence of
  mixed parametric uncertainty and unmodeled dynamics,'' \emph{IEEE Trans.
  Automat. Control}, vol.~36, no.~1, pp. 25--38, 1991.

\bibitem{RCT2013b}
\emph{Robust Control Toolbox 5.0}.\hskip 1em plus 0.5em minus 0.4em\relax
  MathWorks, Natick, MA, USA, Sept 2013.

\bibitem{SK12}
C.~W. Scherer and I.~E. K{\"{o}}se, ``Gain-scheduled control synthesis using
  dynamic {$D$}-scales,'' \emph{IEEE Trans. Automat. Control}, vol.~57, no.~9,
  pp. 2219--2234, 2012.

\bibitem{PA06}
D.~Peaucelle and D.~Arzelier, ``Robust {M}ulti-{O}bjective {C}ontrol toolbox,''
  in \emph{Proc.\ IEEE Conf.\ on Computer Aided Control Systems Design},
  Munich, October 2006, pp. 1152--1157.

\bibitem{NSGT99}
R.~H. Nystr{\"o}m, K.~V. Sandstr{\"o}m, T.~K. Gustafsson, and H.~T. Toivonen,
  ``Multimodel robust control of nonlinear plants: a case study,'' \emph{J.
  Process Contr.}, vol.~9, no.~2, pp. 135--150, 1999.

\bibitem{MGC98}
J.-F. Magni, Y.~{Le Gorrec}, and C.~Chiappa, ``A multimodel-based approach to
  robust and self-scheduled control design,'' in \emph{Proc.\ IEEE Conf.\ on
  Decision and Control}, vol.~3, 1998, pp. 3009--3014.

\bibitem{ABKSS93}
J.~Ackermann, A.~Bartlett, D.~Kaesbauer, W.~Sienel, and R.~Steinhauser,
  \emph{Robust control. Systems with Uncertain Physical Parameters}, ser. Comm.
  Control Engrg. Ser.\hskip 1em plus 0.5em minus 0.4em\relax London:
  Springer-Verlag London, Ltd., 1993.

\bibitem{Cla83}
F.~H. Clarke, \emph{Optimization and Nonsmooth Analysis}, ser. Canad. Math.
  Soc. Ser. Monogr. Adv. Texts.\hskip 1em plus 0.5em minus 0.4em\relax New
  York: John Wiley \& Sons, Inc., 1983.

\bibitem{ZDG96}
K.~Zhou, J.~C. Doyle, and K.~Glover, \emph{Robust and Optimal Control}.\hskip
  1em plus 0.5em minus 0.4em\relax New Jersey: Prentice Hall, 1996.

\bibitem{Red60}
R.~M. Redheffer, ``On a certain linear fractional transformation,'' \emph{J.
  Math. and Phys.}, vol.~39, pp. 269--286, 1960.

\bibitem{BHLO06}
J.~V. Burke, D.~Henrion, A.~S. Lewis, and M.~L. Overton, ``{HIFOO} - {A Matlab}
  package for fixed-order controller design and {$H_\infty$} optimization,'' in
  \emph{5th IFAC Symposium on Robust Control Design}, Toulouse, July 2006.

\bibitem{Spi81}
J.~E. Spingarn, ``Submonotone subdifferentials of {L}ipschitz functions,''
  \emph{Trans. Amer. Math. Soc.}, vol. 264, no.~1, pp. 77--89, 1981.

\bibitem{RW98}
R.~T. Rockafellar and R.~J.-B. Wets, \emph{Variational Analysis}.\hskip 1em
  plus 0.5em minus 0.4em\relax Berlin: Springer-Verlag, 1998.

\bibitem{ANP08}
P.~Apkarian, D.~Noll, and O.~Prot, ``A proximity control algorithm to minimize
  nonsmooth and nonconvex semi-infinite maximum eigenvalue functions,''
  \emph{J. Convex Anal.}, vol.~16, no. 3-4, pp. 641--666, 2009.

\bibitem{ANP07}
------, ``A trust region spectral bundle method for nonconvex eigenvalue
  optimization,'' \emph{SIAM J. Optim.}, vol.~19, no.~1, pp. 281--306, 2008.

\bibitem{BBK89}
S.~Boyd, V.~Balakrishnan, and P.~Kabamba, ``A bisection method for computing
  the {$\mathbf H_\infty$} norm of a transfer matrix and related problems,''
  \emph{Math. Control Signals Systems}, vol.~2, no.~3, pp. 207--219, 1989.

\bibitem{BBal90}
S.~Boyd and V.~Balakrishnan, ``A regularity result for the singular values of a
  transfer matrix and a quadratically convergent algorithm for computing its
  {$\mathbf L_\infty$}-norm,'' \emph{Systems Control Lett.}, vol.~15, no.~1,
  pp. 1--7, 1990.

\bibitem{BSV12}
P.~Benner, V.~Sima, and M.~Voigt, ``{$\mathcal L_\infty$}-norm computation for
  continuous-time descriptor systems using structured matrix pencils,''
  \emph{IEEE Trans. Automat. Control}, vol.~57, no.~1, pp. 233--238, 2012.

\bibitem{BB91}
S.~Boyd and C.~Barratt, \emph{Linear Controller Design: Limits of
  Performance}.\hskip 1em plus 0.5em minus 0.4em\relax New York: Prentice Hall,
  1991.

\bibitem{BO94}
J.~V. Burke and M.~L. Overton, ``Differential properties of the spectral
  abscissa and the spectral radius for analytic matrix-valued mappings,''
  \emph{Nonlinear Anal.}, vol.~23, no.~4, pp. 467--488, 1994.

\bibitem{BAN07}
V.~Bompart, P.~Apkarian, and D.~Noll, ``Non-smooth techniques for stabilizing
  linear systems,'' in \emph{Proc.\ American Control Conf.}, New York, July
  2007, pp. 1245--1250.

\bibitem{Lui11}
S.~H. Lui, ``Pseudospectral mapping theorem {II},'' \emph{Electron. Trans.
  Numer. Anal.}, vol.~38, pp. 168--183, 2011.

\bibitem{Noll12}
D.~Noll, ``Convergence of non-smooth descent methods using the
  {Kurdyka-{\L}ojasiewicz} inequality,'' \emph{J. Optim. Theory Appl.}, vol.
  160, no.~2, pp. 553--572, 2014.

\bibitem{Ber82}
D.~P. Bertsekas, \emph{Constrained optimization and {L}agrange multiplier
  methods}, ser. Comput. Sci. Appl. Math.\hskip 1em plus 0.5em minus
  0.4em\relax New York-London: Academic Press, Inc., 1982.

\bibitem{Kiw83}
K.~C. Kiwiel, ``An aggregate subgradient method for nonsmooth convex
  minimization,'' \emph{Math. Programming}, vol.~27, no.~3, pp. 320--341, 1983.

\bibitem{Dao}
M.~N. Dao, ``Bundle method for nonconvex nonsmooth constrained optimization,''
  2014, submitted.

\bibitem{Noll10}
D.~Noll, ``Cutting plane oracles to minimize non-smooth non-convex functions,''
  \emph{Set-Valued Var. Anal.}, vol.~18, no. 3-4, pp. 531--568, 2010.

\bibitem{BDL06}
J.~Bolte, A.~Daniilidis, and A.~Lewis, ``The {{\L}ojasiewicz} inequality for
  nonsmooth subanalytic functions with applications to subgradient dynamical
  systems,'' \emph{SIAM J. Optim.}, vol.~17, no.~4, pp. 1205--1223, 2006.

\bibitem{BM88}
E.~Bierstone and P.~D. Milman, ``Semianalytic and subanalytic sets,''
  \emph{Inst. Hautes \'{E}tudes Sci. Publ. Math.}, vol.~67, pp. 5--42, 1988.

\bibitem{KK02}
D.~Klatte and B.~Kummer, \emph{Nonsmooth Equations in Optimization. Regularity,
  Calculus, Methods and Applications}, ser. Nonconvex Optim. Appl.\hskip 1em
  plus 0.5em minus 0.4em\relax Dordrecht: Kluwer Academic Publishers, 2002,
  vol.~60.

\bibitem{Mif77}
R.~Mifflin, ``Semismooth and semiconvex functions in constrained
  optimization,'' \emph{SIAM J. Control Optimization}, vol.~15, no.~6, pp.
  959--972, 1977.

\bibitem{MBO97}
J.~Moro, J.~V. Burke, and M.~L. Overton, ``On the
  {L}idskii-{V}ishik-{L}yusternik perturbation theory for eigenvalues of
  matrices with arbitrary {J}ordan structure,'' \emph{SIAM J. Matrix Anal.
  Appl.}, vol.~18, no.~4, pp. 793--817, 1997.

\bibitem{BLO02}
J.~V. Burke, A.~S. Lewis, and M.~L. Overton, ``Two numerical methods for
  optimizing matrix stability,'' \emph{Linear Algebra Appl.}, vol. 351-352, pp.
  117--145, 2002, fourth special issue on linear systems and control.

\bibitem{BLO05}
------, ``A robust gradient sampling algorithm for nonsmooth, nonconvex
  optimization,'' \emph{SIAM J. Optim.}, vol.~15, no.~3, pp. 751--779, 2005.

\bibitem{AGB14}
P.~Apkarian, P.~Gahinet, and C.~Buhr, ``Multi-model, multi-objective tuning of
  fixed-structure controllers,'' in \emph{European Control Conf.\ (ECC)},
  Strasbourg, June 2014.

\bibitem{DFT92}
J.~C. Doyle, B.~A. Francis, and A.~R. Tannenbaum, \emph{Feedback Control
  Theory}.\hskip 1em plus 0.5em minus 0.4em\relax New York: Macmillan
  Publishing Company, 1992.

\bibitem{Res91}
C.~S. Resnik, ``A method for robust control of systems with parametric
  uncertainty motivated by a benchmark example,'' Master's thesis, June 1991.

\bibitem{CK09}
U.~Chaiya and S.~Kaitwanidvilai, ``Fixed-structure robust {DC} motor speed
  control,'' in \emph{Proc.\ International MultiConference of Engineers and
  Computer Scientists (IMECS)}, vol.~II, Hong Kong, March 2009, pp. 1533--1536.

\bibitem{FSBS03}
G.~Filardi, O.~Sename, A.~Besancon-Voda, and H.-J. Schroeder, ``Robust
  {$H_\infty$} control of a {DVD} drive under parametric uncertainties,'' in
  \emph{European Control Conf.\ (ECC)}, Cambridge, September 2003.

\bibitem{Enn84}
D.~F. Enns, ``Model reduction for control system design,'' Ph.D. dissertation,
  Stanford University, 1984.

\bibitem{VGD01}
R.~Vadigepalli, E.~P. Gatzke, and F.~J. {Doyle III}, ``Robust control of a
  multivariable experimental four-tank system,'' \emph{Ind. Eng. Chem. Res.},
  vol.~40, no.~8, pp. 1916--1927, 2001.

\bibitem{GPK05}
D.~W. Gu, P.~H. Petkov, and M.~M. Konstantinov, \emph{Robust Control Design
  with Matlab}.\hskip 1em plus 0.5em minus 0.4em\relax London: Springer-Verlag,
  2005.

\bibitem{CD94}
Y.~Cheng and B.~L. R.~D. Moor, ``Robustness analysis and control system design
  for a hydraulic servo system,'' \emph{IEEE Trans. on Control System
  Technology}, vol.~2, no.~3, pp. 183--197, 1994.

\bibitem{ACAGF99}
D.~Alazard, C.~Cumer, P.~Apkarian, M.~Gauvrit, and G.~Ferreres,
  \emph{Robustesse et Commande Optimale}.\hskip 1em plus 0.5em minus
  0.4em\relax Toulouse: {C\'{e}padu\`{e}s \'{E}ditions}, 1999.

\bibitem{Kru93}
D.~L. Krueger, ``Parametric uncertainty reduction in robust multivariable
  control,'' Ph.D. dissertation, Naval Postgraduate School, September 1993.

\bibitem{SK08}
C.~W. Scherer and I.~E. K{\"{o}}se, ``Robustness with dynamic {IQCs}: an exact
  state-space characterization of nominal stability with applications to robust
  estimation,'' \emph{Automatica J. IFAC}, vol.~44, no.~7, pp. 1666--1675,
  2008.

\bibitem{Kim06}
Y.-M. Kim, ``Robust and reduced order {H-Infinity} filtering via {LMI} approach
  and its application to fault detection,'' Ph.D. dissertation, Wichita State
  University, May 2006.

\bibitem{NTA04}
D.~Noll, M.~Torki, and P.~Apkarian, ``Partially augmented {L}agrangian method
  for matrix inequality constraints,'' \emph{SIAM J. Optim.}, vol.~15, no.~1,
  pp. 161--184, 2004.

\bibitem{Young96}
P.~M. Young, ``Controller design with real parametric uncertainty,''
  \emph{Internat. J. Control}, vol.~65, no.~3, pp. 469--509, 1996.

\end{thebibliography}

\end{document}